\documentclass[11pt,english,reqno]{amsart}
\usepackage[%
   pdftitle={Spectral properties of general streaming operators and semigroups},
    pdfauthor={Bertrand Lods, M. Mokhtar--Kharroubi and M. Sbihi},
    pdfsubject={Spectral properties of general streaming operators and semigroups},
    pdfkeywords={},
   colorlinks=true,
   linkcolor=red,
   citecolor=blue,
 ]{hyperref}
\usepackage[T1]{fontenc}
\usepackage{mathrsfs}
\usepackage{amsmath,amsthm,amssymb}
\usepackage{latexsym}
\usepackage{times}
\usepackage{enumerate}
\usepackage{mathrsfs}
\usepackage{stmaryrd}
\usepackage{amsopn}
\usepackage{amsmath}
\usepackage{amssymb}
\usepackage{amsfonts}
\usepackage{amsbsy}
\usepackage{amscd,indentfirst}
\usepackage{hyperref}
\usepackage{amsfonts,amsmath,latexsym,amssymb,verbatim,amsbsy}
\usepackage{amsthm}
\usepackage{colordvi}
\usepackage{dsfont}

\usepackage{pstricks} \hsize=126mm \vsize=180mm
\newtheorem{theo}{Theorem}[section]

\newtheorem{propo}[theo]{Proposition}
\newtheorem{cor}[theo]{Corollary}
\newtheorem{hyp}[theo]{Assumption}
\newtheorem{nb}[theo]{Remark}
\newtheorem{defi}[theo]{Definition}

\theoremstyle{definition}

\newtheorem{exa}[theo]{Example}
\def \leq {\leqslant}
\def \geq {\geqslant}
\numberwithin{equation}{section}
\def \k {n}

\newtheorem{theoA}{Theorem A.}
\newtheorem{propoA}[theoA]{Proposition A.}
\newtheorem{nbA}[theoA]{Remark A.}
\def\ind#1{\lower5pt\hbox{$\scriptstyle #1$}}

\def \x {\mathbf{x}}

\def\y {\mathbf{y}}

\def \ff {\mathscr{F}}
\def \d {\mathrm{d}}
\def \D {\mathscr{D}}

\def \R {\mathbb{R}}

\def \F {\mathbf{F}}

\def \uot {(\u (t))_{t \geq 0}}
\def \ut {(T(t))_{t \geq 0}}

\def \esin {\operatornamewithlimits{ess\,inf}}

\def \ds {\displaystyle}
\def \T {\mathcal{T}}

\def \t {\tau}
\def \u {\mathcal{U}}
\def \O {\mathbf{\Omega}}
\def \X {\mathfrak{X}}

\renewcommand {\sigma}{\mathfrak{S}}
\title[Spectral properties of general advection operators and semigroups]{Spectral properties of general advection operators and weighted translation semigroups}

\author{B. \textsc{Lods}}
\address{Laboratoire de Math\'{e}matiques, CNRS UMR 6620, Universit\'{e}
Blaise Pascal (Clermont-Ferrand 2), 63177 Aubi\`{e}re Cedex, France.
{\tt bertrand.lods@math.univ-bpclermont.fr} }
\author{M. \textsc{Mokhtar-Kharroubi}}
\address{ Universit\'e de Franche--Comt\'e, Laboratoire de Math\'ematiques, CNRS UMR
6623, 16, route de Gray, 25030 Besan\c con Cedex, France.
\noindent{\tt mmokhtar@univ-fcomte.fr}}
\author{M. \textsc{Sbihi}}
\address{ Universit\'e de Franche--Comt\'e, Laboratoire de Math\'ematiques, CNRS UMR
6623, 16, route de Gray, 25030 Besan\c con Cedex, France.
\noindent{\tt msbihi@univ-fcomte.fr}}
\begin{document}
\thanks{{\it Keywords:} Spectral Mapping Theorem, weighted shift,
Annular Hull Theorem, transport equation.\\
}
                             \maketitle
                                            \begin{abstract}
We investigate the spectral properties of a class of weighted shift
semigroups $\uot$ associated to abstract transport equations with a
Lipschitz--continuous vector field $\ff$ and no--reentry boundary
conditions. Generalizing the results of \cite{positivity}, we prove
that the semigroup $\uot$ admits a canonical decomposition into
three $C_0$-semigroups $(\u_1(t))_{t \geq 0}$, $(\u_2(t))_{t \geq
0}$ and $(\u_3(t))_{t \geq 0}$ with independent dynamics. A complete
description of the spectra of the semigroups $(\u_i(t))_{t \geq 0}$
and their generators $\T_i$, $i=1,2$ is given. In particular, we
prove that the spectrum of $\T_i$ is a left-half plane and that the
Spectral Mapping Theorem holds: $\sigma(\u_i(t))=\exp\left\{ t
\sigma(\T_i)\right\}$, $i=1,2$. Moreover, the semigroup
$(\u_3(t))_{t \geq 0}$ extends to a $C_0$-group and its spectral
properties are investigated by means of abstract results from
positive  semigroups theory. The properties of the flow associated
to $\ff$ are particularly relevant here and we investigate
separately the cases of periodic and aperiodic flows. In particular,
we show that, for periodic flow, the Spectral Mapping Theorem fails
in general but $(\u_3(t))_{t \geq 0}$ and its generator $\T_3$
satisfy the so-called Annular Hull Theorem. We illustrate our
results with various examples taken from collisionless kinetic
theory.
\end{abstract}
\medskip                         
%
%
%
\section{Introduction and preliminaries}

We develop in the present paper a systematic approach to the
spectral analysis in $L^p$-spaces $(1 \leq p <\infty)$  of a class of \textit{weighted shift semigroups} arising in
kinetic theory
\begin{equation}\label{defiut}
\u(t)\::f \longmapsto \u(t)f(\x)=\exp{\left[-\int_0^t
\nu(\Phi(\x,-s))\d s\right]}f(\Phi(\x,-t))\chi_{\{t <
\tau_-(\x)\}}(\x)\end{equation} where the flow $\Phi(\x,t)$ is
associated to a globally Lipschitz transport field $\ff$ and
$\nu(\x)$ is given by
$$\nu(\x)=h(\x) + \mathrm{div}(\ff)(\x), \qquad \x \in \O$$
where the functions $h$ and $\ff$ satisfy the following:
\begin{hyp}  The field $ \mathscr {F} \::\:\overline{\mathbf{\Omega}} \to
\mathbb{R}^N$ is Lipschitz-continuous with Lipschitz constant
$\kappa > 0.$ Moreover, its divergence $ \mathrm{div}(\ff)$ is a
bounded function on $\O$.  The absorption function $h(\cdot)$ is
measurable and bounded below.
 \end{hyp}
Before explaining more in details the contents of this work, we have
to explicit a bit the first properties of the different terms
arising in Eq. \eqref{defiut}. Since $\ff$ is Lipschitz over
$\mathbf{\Omega}$ (with constant $\kappa >0$), it is known from
Kirszbraun's extension theorem \cite[p. 201]{federer}, that $\ff$
can be extended as a Lipschitz function (with the same Lipschitz
constant $\kappa>0$)  over the whole space $\mathbb{R}^N$. We shall
still denote this extension by $\ff$. In Eq. \eqref{defiut},
$\Phi(\x,t)$ is the unique maximal solution of the characteristic
equation
\begin{equation}\label{chara}
\begin{cases}
\frac{\d }{\d t}\mathbf{X}(t)=\ff(\mathbf{X}(t)), \qquad (t \in \mathbb{R});\\
\mathbf{X}(0)=\x,
\end{cases}\end{equation}
which is well-defined since  the (extended) field $\ff$ is globally
Lipschitz. In \eqref{defiut}, $\tau_-(\x)$ denotes the stay time in
$\O$ of the characteristic curves $t >0 \mapsto \Phi(\x,-t)$
starting from $\x \in \O$ :
\begin{equation}\label{tau}\tau_{\pm}(\x)=\inf \{s
> 0\,;\Phi(\x,\pm s) \notin \mathbf{\Omega}\},\end{equation} with the convention
that $\inf \varnothing=\infty.$ In other words, given $\x \in \O$,
$I_{\x}=(-\t_-(\x),\t_+(\x))$ is the maximal time interval for which
the solution $\mathbf{X}(t)$ lies in $\mathbf{\Omega}$ for any $t
\in I_{\x}$. We shall denote by $\t(\x):=\t_+(\x)+\t_-(\x)$ the
length of the maximal interval $I_\x$.

The general strategy we adopt to describe the spectral
properties of the semigroup $(\u(t))_{t \geq 0}$ consists in a canonical
decomposition of the semigroup $(\u(t))_{t \geq 0}$ into three semigroups $ (\u_i(t))_{t \geq 0}$,
$(i=1,2,3)$ with \textit{independent dynamics}, the third one
$(\u_3(t))_{t \geq 0}$ extending to a $C_0$-group. Notice that it would be possible to investigate the spectral properties of $(\u_3(t))_{t \in \mathbb{R}}$ within the general framework  developed in
\cite[Chapter 6]{chicone} (see also \cite{lst}). The approach of \cite{chicone} uses sophisticated tools from dynamical systems theory while our approach is completely different and relies on general results concerning the spectral properties of $C_0$-groups in $L^p$-spaces given in the Appendix.

\subsection{Preliminaries and motivation}

If $\tau(\x)$ is finite, then the function $\mathbf{X}\::\:s \in
I_{\x} \longmapsto \Phi(\x,s)$ is bounded since $\ff$ is Lipschitz
continuous. Moreover, still by virtue of the Lipschitz continuity of
$\ff$, the only case when $\t_{\pm}(\x)$ is finite is when $\Phi(\x,\pm
s)$ reaches the boundary $\partial\mathbf{\Omega}$ so that
$\Phi(\x,\pm \tau_{\pm}(\x)) \in
\partial\mathbf{\Omega}$.
We finally mention that it is not difficult to prove that the
mappings $\tau_{\pm}$ : $\mathbf{\Omega} \to \mathbb{R}^+$ are lower
semi-continuous and therefore measurable \cite[p. 301]{arloban}.
Note that, since the field $\ff$ is not assumed to be
divergence--free, then the transformation induced by the flow $\Phi$
is not measure--preserving. Precisely, one has the following
\cite{lasota, m2as}:
\begin{propo}\label{mea} Let $\varrho_t$ denote the image of the
  $N$-dimensional Lebesgue measure  $\mathfrak{m}$
through the transformation $T_t:\x \mapsto \Phi(\x,-t)$, $(t \in
I_\x)$. Then,  $\varrho_t$ is absolutely continuous with respect to
$\mathfrak{m}$, and its Radon-Nikodym derivative
$\frac{\d\varrho_t}{\d\mathfrak{m}}$ with respect to $\mathfrak{m}$
is given by
$$\dfrac{\d\varrho_t}{\d\mathfrak{m}}(\x)=\exp\left[\int_0^t \mathrm{div}(\ff)(\Phi(\x,s))\d s\right]
\quad \text{ for } \mathfrak{m}-\text{a.e. } \:\x \in \O,\;t \in
I_\x.$$
\end{propo}

The semiflow $\Phi$ enjoys the following elementary properties
\cite{arloban}
\begin{propo}\label{Phiprop} Let $\x \in \mathbf{\Omega}$ and $t \in \mathbb{R}$ be fixed. Then,
\begin{enumerate}[(i)\:]
\item $\Phi(\x,0)=\x.$ \item
$\Phi(\Phi(\x,s_1),s_2)=\Phi(\x,s_1+s_2), \quad \forall s_1 \in
I_\x, s_2 \in (-s_1-\t_-(\x), \t_+(\x)-s_1).$ \item
$\left|\Phi(\x_1,t)-\Phi(\x_2,t )\right| \leq \exp(\kappa |t
|)|\x_1-\x_2|$ for any $\x_1,\x_2 \in \mathbf{\Omega}$, $t \in
I_{\x_1} \cap I_{\x_2}.$
\end{enumerate}\end{propo}

In all the paper, we \textit{fix} $1 \leq p < \infty$ and set
$$X=L^p(\O,\d\mathfrak{m}).$$
The following classical result (see, e.g.,
\cite{arloban,m2as,bardos}) asserts that the family $(\u(t))_{t \geq
0}$ given by \eqref{defiut} is a strongly continuous semigroup of
bounded operators in $X$.
\begin{theo}
\label{ut}
 Let
$$\u(t)f(\x)=\exp{\left[-\int_0^t \nu(\Phi(\x,-s))\d s\right]}f(\Phi(\x,-t))\chi_{\{t <
\tau_-(\x)\}}(\x), \qquad \x \in \O,\:f \in X ,$$ where $\chi_A$
denotes the characteristic function of a set $A$. The family
$(\u(t))_{t \geq 0}$ is a positive $C_0$-semigroup of bounded
operators in $X$. We shall denote by $(\T,\D(\T))$ its generator.
\end{theo}

In the present paper, we do not need to explicit further the
generator $(\T,\D(\T))$ of $(\u(t))_t$ since our spectral analysis
does not depend on its description. Note however  that, if $\O$ is a
sufficiently smooth open subset of $\R^N$, then the generator
$(\T,\D(\T))$ is explicitly described in \cite{m2as} for $h=0$. In
some sense, which we do not explicit here (see for instance
\cite{m2as,abl}), the semigroup $(\u(t))_{t \geq 0}$ governs the
following advection equation:
\begin{subequations}\label{1}
\begin{equation}\label{1a}
\partial_t f(\x,t)+\nabla_\x\cdot\left(\mathscr{F}(\x)
f(\x,t)\right)+h(\x)f(\x,t) =0 \qquad (\x \in \O, \:t
> 0),\end{equation} supplemented by the boundary condition
\begin{equation}\label{1b}
f_{|\Gamma_-}(\y,t)=0, \qquad \qquad (\y \in \Gamma_-, t >0),
\end{equation}
and the initial condition
\begin{equation}\label{1c}f(\x,0)=f_0(\x), \qquad \qquad (\x \in \O), \end{equation}\end{subequations}
where $\Gamma_-$ is the incoming part of the boundary of $\O$ (we
refer the reader to \cite{bardos,arloban} for details on the
matter), i.e. $f(\x,t)=[\u(t)f_0](\x)$ at least for regular initial
data $f_0$. A typical example of such an absorption equation is the
so-called Vlasov equation for which:
\begin{enumerate}[i)\:] \item The phase space $\O$ is given by the
cylindrical domain $\O= {\mathcal{D}}\times \mathscr{V} \subset
\mathbb{R}^6$ where $\mathcal{D}$ is a smooth open subset of
$\mathbb{R}^3$, referred to as the \textit{position space}, and
$\mathscr{V}\subset \mathbb{R}^3$ is referred to as the
\textit{velocity space}. \item For any $\x=(x,v) \in \mathcal{D}
\times \mathscr{V}$,
\begin{equation}\label{vlasov}\ff(\x)=(v,\F(x,v)) \in \mathbb{R}^6\end{equation} where
$\F=(\F_1,\F_2,\F_3)$ is a time independent globally Lipschitz field
(the force field) over $\mathcal{D} \times \mathscr{V}$.
\end{enumerate}

With the above choice, Eq. \eqref{1a} reads:
\begin{equation}\label{truevlasov}
\partial_t f(x,v,t) + v \cdot \nabla_x f(x,v,t)+
\mathbf{F}(x,v) \cdot \nabla_v f(\x,v,t)+ \nu(x,v)f(x,v,t)=0,
\end{equation}
where $\nu(x,v)=h(x,v)+\mathrm{div}_v (\F)(x,v)$, supplemented with
suitable initial and boundary conditions (see \eqref{1b}). More
general problems can be handled with. For instance, one can treat
with our formalism the collisionless version of the linear Boltzmann
arising in \textit{semiconductor theory}:
$$\partial_t f(\mathbf{r},\mathbf{k},t) +
\dfrac{1}{\hslash} \nabla_{\mathbf{k}} \epsilon(\mathbf{k}) \cdot
\nabla_{\mathbf{k}}f(\mathbf{r},\mathbf{k},t) +\dfrac{e}{\hslash}
\mathcal{E}\cdot \nabla_\mathbf{k}f(\mathbf{r},\mathbf{k},t) +
\nu(\mathbf{r},\mathbf{k})f(\mathbf{r},\mathbf{k},t)=0$$ where the
unknown $f(\mathbf{r},\mathbf{k},t)$ is the density of electrons
having the position $\mathbf{r} \in \mathbb{R}^3,$ the wave-vector
$\mathbf{k} \in \mathbb{R}^3$ at time $t > 0.$ The parameters $e$
and $2\pi\hslash$ are the positive electron charge and the Planck
constant respectively while $\epsilon(\mathbf{k})$ represents the
electron energy and $\mathcal{E}=\mathcal{E}(\mathbf{r},\mathbf{k})$
is an external electric field.

The abstract equation \eqref{1} allows also to consider
collisionless kinetic equation for relativistic models for which
$$v=v(p)=\frac{p}{m\sqrt{1+p^2/c^2m^2}},$$ where $m$ stands for the
mass of particles and $c$ being the velocity of light. For all this
kind of models, the solution $f$ to the collisionless kinetic
equation is given, under suitable conditions on the data, by
$f(\x,t)=\u(t)f_0(\x)$ where $f_0$ corresponds to the initial state
of the system and $\u(t)$ is a weighted shift semigroup of the shape
\eqref{defiut}. We provide in the rest of the paper a large number
of examples arising in kinetic theory for which our abstract results
apply. Spectral properties of systems of scalar advective equations
on the torus coupled by a pseudo-differential operator of order zero
(motivated by fluid mechanics ) are dealt with by Shvydkoy
\cite{shvydkoy} (see also
\cite{shvydkoy1,shvydkoy2,shvydkoy3,chicone2}). We note certain
similarities between some of those results and our results
concerning  the group $(\u_3(t))_{t \in \R}$.

\subsection{Main results and methodology} To describe the spectral
features of the semigroup $(\u(t))_{t \geq 0}$ and its generator
$\T$, we generalize the approach initiated in a recent work of the
second author \cite{positivity}. The analysis of \cite{positivity}
is restricted to the neutron transport equation, corresponding to
the choice of $\F=0$ in the above equation \eqref{truevlasov}. We
generalize \cite{positivity} to general vector field $\ff$.
Precisely, thanks to a suitable decomposition of the phase space
$\O$ according to the finiteness of $\t_-(\cdot)$ and $\t_+(\cdot)$,
we show that a general weighted shift semigroup $(\u(t))_{t \geq 0}$
admits a canonical decomposition into three semigroups $
(\u_i(t))_{t \geq 0}$, $(i=1,2,3)$ with \textit{independent
dynamics}. The third semigroup $(\u_3(t))_{t \geq 0}$ actually
extends to a $C_0$-group, which corresponds to the global in time
flow $\Phi(t,\cdot)$ already investigated in \cite{lst}. Concerning
the spectral properties of the semigroups $ (\u_1(t))_{t \geq 0}$
and $ (\u_2(t))_{t \geq 0}$ (which  correspond to
\textit{trajectories $\left(\Phi(\x,t)\right)_t$ such that either
$\tau_+(\x)$ or $\tau_-(\x)$ is finite}), they both enjoy the same
spectral structure:
\begin{equation}\label{u1u2}\sigma(\u_i(t))=\{\mu\,;\,|\mu| \leq \exp(-\gamma_i
t)\}, \qquad \sigma(\T_i)=\{\lambda\,;\,\mathrm{Re}\lambda \leq
-\gamma_i\}, \quad i=1,2\end{equation} where $\gamma_i$ $(i=1,2)$
are positive constants, depending on $h$ and $\ff$ (see Theorems
\ref{Spectu1} \& \ref{Spectu2}). The above description  relies on
several abstract results on positive semigroups on Banach lattices
and, in particular, on the following property, proved  in
\cite{positivity}:
\begin{propo}\label{muspos} Let $\ut$ be a $C_0$-semigroup of positive operators on a complex Banach lattice $\X$.
Let $\mathfrak{Y}$ denote the subspace of local quasinilpotence of
$\ut$: $$\mathfrak{Y}=\left\{u \in \X\,;\,\lim_{t \to
\infty}\exp\left[\dfrac{1}{t}\log{\|T(t)(|u|)\|_{\X}}\right]=0\right\}.$$
If  $\mathfrak{Y}$ is dense in $\X$ then  $[0,\exp(\omega_0(T) t)]
\subset \sigma_{\mathrm{ap}}(T(t))$ for any $t \geq 0 $ while
$(-\infty,s(A)] \subset \sigma_{\mathrm{ap}}(A)$, where $\omega_0(T)
\in [-\infty,\infty)$ is the type of $\ut$ and $s(A)$ denotes the
spectral bound of the generator $A$ of $\ut$.
\end{propo}
Such a result allows to describe very precisely the \textit{real
spectrum} of $\u_1(t)$ and $\u_2(t)$. Then, to deal with the
non-real spectrum, we prove the invariance by rotation of the
spectrum of $\u_1(t)$ and $\u_2(t)$ and the invariance  by
translation along the imaginary axis of their generators. This is
done in the spirit of \cite{voigt}, see Proposition \ref{propvoigt}.

 The description \eqref{u1u2} shows in particular that, whenever the flow
$\Phi$ (and the geometry $\O$) do not allow trajectories that are
global in both positive and negative times (i.e. either $\tau_+(\x)$
or $\t_-(\x)$ is finite), then the spectrum of the generator $\T$ is
a left-half plane and the Spectral Mapping Theorem
\begin{equation}\label{SPMT}\sigma(\u(t)) \setminus \{0\}=\exp\left(\sigma(\T) t \right), \qquad t \geq
0,\end{equation} holds. Such a result seems to be new. Actually, we
show that only the existence of periodic orbits and/or stationary
points could prevent the Spectral Mapping Theorem \eqref{SPMT} to
hold. Indeed, when dealing with the $C_0$-group $(\u_3(t))_{t \in
\mathbb{R}}$, we show that, here again, this group admits a
canonical decomposition into three groups with independent dynamics
$(\u_{\mathrm{rest}}(t))_{t \in \mathbb{R}}$,
$(\u_{\mathrm{per}}(t))_{t \in \mathbb{R}}$  and
$(\u_{\infty}(t))_{t \in \mathbb{R}}$ corresponding respectively to
stationary orbits, periodic orbits and infinite orbits which are
neither stationary nor periodic. The spectral analysis of
$(\u_\mathrm{rest}(t))_{t \in \mathbb{R}}$ is really easy to derive
since $\u_\mathrm{rest}(t)$ acts as a \textit{multiplication} group.
On the other hand, the group $(\u_{\infty}(t))_{t \in \mathbb{R}}$
falls within the general theory of Mather groups associated to
aperiodic flow investigated in \cite{chicone} and \cite{lst}.
Concerning the delicate case of periodic trajectories, we deal with
a description of the spectrum of the generator $\T_\mathrm{per}$ and
prove (thanks to a general result on positive $C_0$-\textit{groups}
on $L^p$-spaces) that $(\u_{\mathrm{per}}(t))_{t \in \mathbb{R}}$
fulfils the so-called Annular Hull Theorem
$$\mathbb{T} \cdot \sigma(\u_\mathrm{per}(t)) = \mathbb{T} \cdot
\exp\bigg(t \big(\sigma(\T_\mathrm{per}) \cap \R\big)\bigg), \qquad
\forall t \in \mathbb{R}.$$ Notice that, in full generality,
$\u_\mathrm{per}(t)$ does not satisfy the Spectral Mapping Theorem
\eqref{SPMT} (see Example \ref{exa:cercle1}).

The organization of the paper is as follows: in Section 2, we
establish the aforementioned decomposition of $(\u(t))_{t \geq 0}$
into three semigroups $ (\u_i(t))_{t \geq 0}$, $(i=1,2,3)$ with
independent dynamics. Moreover, we provide a complete description of
the spectrum of $(\u_1(t)_{t \geq 0}$, $(\u_2(t))_{t \geq 0}$ and
their generators and illustrate our results by several examples from
kinetic theory. The properties of the group $(\u_3(t))_{t \geq 0}$
are investigated in Section 3 where we deal only with stationary or
aperiodic flow. In Section 4, we investigate the more delicate case
of a periodic flow. Finally, in the Appendix, we state some known
and new abstract results on the spectral properties of positive
$C_0$-semigroups in $L^p$-spaces, $1 \leq p  < \infty.$

\section{Spectral properties of the streaming semigroup
$\uot$}\label{sec:u1u2}
 Let us define the
following subsets of $\O$:
$$\O_1=\left\{\x \in \O\,;\,\t_+(\x) < \infty\right\}, \qquad
\O_2=\left\{\x \in \O\,;\,\t_+(\x) = \infty \text{ and } \t_-(\x) <
\infty\right\}$$ and $$\O_3=\left\{\x \in
\O\,;\,\t_+(\x)=\t_-(\x)=\infty\right\}.$$ Moreover, define
$$X_i=\left\{f \in X\,;\,f(\x)=0 \quad \mathfrak{m}-\text{a.e. } \x \in \O \setminus \O_i\right\}, \qquad
i=1,2,3.$$ In the sequel, we shall \textit{identify} $X_i$ with
$L^p(\O_i,\d\mathfrak{m})$, $i=1,2,3$.  Since
$(\O_i)_{i=1,\ldots,3}$ is a partition of $\O$, it is clear that
$$X=X_1 \oplus X_2 \oplus X_3.$$
Of course, if $\mathfrak{m}(\O_i)=0$, the space $X_i$ reduces to
$\{0\}$ and drops out in the direct sum $(i=1,2,3)$. Following
\cite{positivity}, we can state the following:
\begin{theo}\label{decomp} For any $i=1,2,3$, $X_i$ is invariant under $\uot$.
Define $\u_i(t)={\u(t)}_{|X_i}$ for any $t \geq 0$, $i=1,2,3$. Then,
$(\u_i(t))_{t \geq 0}$ is a positive $C_0$-semigroup of $X_i$
$(i=1,2,3)$ and
\begin{equation}\label{reduces}
\sigma(\u(t))=\sigma(\u_1(t)) \cup \sigma(\u_2(t))\cup \sigma
(\u_3(t))\qquad \qquad (t \geq 0)\end{equation} where $\sigma
(\u_i(t))$ stands for the spectrum of $\u_i(t)$ in the space $X_i$
($i=1,2,3$). Moreover, $(\u_3(t))_{t \geq 0}$ extends to a
$C_0$-group in $X_3$.
\end{theo}
\begin{proof} Let $f \in X$ and recall that
$$\u(t)f(\x)=\exp{\left[-\int_0^t \nu(\Phi(\x,-s))\d s\right]}f(\Phi(\x,-t))\chi_{\{t <
\tau_-(\x)\}}(\x) \qquad (t \geq 0).$$ Consequently, if $\u(t)f(\x)
\neq 0$, then $t < \t_-(\x)$ and $f(\Phi(\x,-t)) \neq 0$. Moreover,
one deduces easily from Proposition \ref{Phiprop}, that
$\t_+(\Phi(\x,-t))=t+\t_+(\x)$ for any $t < \t_-(\x)$. Therefore,
$\tau_+(\x) < \infty$ if and only if $\tau_+(\Phi(\x,-t))$ for any
$t < \tau_-(\x)$. As a direct consequence, one gets that
$\u(t)f(\x)=0$ $\mathfrak{m}$-a.e. on $\O \setminus \O_1$ for any $f
\in X_1$. This shows that $X_1$ is invariant under $\uot$. In the
same way, still using Proposition \ref{Phiprop}, one observes that
$$\t_-(\Phi(\x,-t))=\tau_-(\x)-t \qquad \forall t <\t_-(\x)$$
so that $\tau_-(\x) < \infty$ if and only if $\tau_-(\Phi(\x,-t))$
for any $t < \tau_-(\x)$. As before, this leads to the invariance of
both $X_2$ and $X_3$ under $\uot$. Thus, the triplet $(X_1,X_2,X_3)$
reduces $\uot$ and  \eqref{reduces} follows. Finally, defining
$\u_3(-t)f(\x)$ for any $t \geq 0$, $\x \in \O_3$ and $f \in X_3$ as
\begin{equation*}\u_3(-t)f(\x)=\exp{\left[\int_{-t}^0 \nu(\Phi(\x,-s))\d
s\right]}f(\Phi(\x,t))\end{equation*}
 one obtains easily the last assertion.
\end{proof}

From the above Theorem, to describe the spectra of $\ut$ and $\T$,
one can deal separately with the spectral properties of the various
semigroups $(\u_i(t))_{t \geq 0}$ and their generator $\T_i$ on
$X_i$, $i=1,2,3.$ We will show in the rest of this paper that our
analysis of $(\u_1(t))_{t \geq 0}$ and $(\u_2(t))_{t \geq 0}$
differs very much from that of the group $(\u_3(t))_{t \in \R}$. In
particular, to show the rotational invariance of the spectra of the
formers, we will make use of the following  result \cite{voigt}:
\begin{propo} \label{propvoigt} Let $\widetilde{\O}$ be a measurable subset of
$\O$ such that $$\widetilde{X}=\left\{f \in X\,;\,f(\x)=0 \quad
\mathfrak{m}-\text{a.e. } \x \in \O \setminus \widetilde{\O}\right\}$$
is invariant under $\uot$ and let $(\widetilde{\u}(t))_{t \geq 0}$
and $\widetilde{\T}$ be the restrictions to $\widetilde{X}$ of
$\uot$ and $\T$  respectively. Assume that there exists a
\textbf{measurable} mapping $\alpha(\cdot)\::\:\O \to \mathbb{R}$
such that
\begin{enumerate}[(a)\:] \item $|\alpha(\x)|$ is finite for almost
any $\x \in \widetilde{\O}$; \item for almost any $\x \in \widetilde{\O}$
such that $t < \t_-(\x)$,
\begin{equation}\label{alpha}
\alpha(\Phi(\x,-t))=\alpha(\x) + t  .\end{equation}
\end{enumerate}
Then, for any $\eta \in \mathbb{R}$, the mapping
$$\mathcal{M}_{\eta}:\:\:f
\in \widetilde{X} \longmapsto [\mathcal{M}_{\eta}f](\x)=\exp(-i\eta
\alpha(\x))f(\x) \in \widetilde{X}$$  is an isometric isomorphism of
$\widetilde{X}$ such that
\begin{equation}\label{meta}
\mathcal{M}_{\eta}^{-1}\widetilde{\T}
\mathcal{M}_{\eta}=\widetilde{\T} +i\eta\mathrm{Id} \qquad \text{
and } \qquad
\mathcal{M}_{\eta}^{-1}\widetilde{\u}(t)\mathcal{M}_{\eta}=e^{i\eta
t}\widetilde{\u}(t) \qquad (t \geq 0).\end{equation} Consequently,
$\sigma(\widetilde{\T})=\sigma(\widetilde{\T}) + i \mathbb{R}$ and $
\sigma(\widetilde{\u}(t))=\sigma(\widetilde{\u}(t))\cdot \mathbb{T}$
for any $(t \geq 0)$ where $\mathbb{T}$ is the unit circle of the
complex plane.
\end{propo}

\subsection{Spectral properties of $(\u_1(t))_{t \geq
0}$}\label{subsec:u1}

To investigate the spectral properties of the restriction of the
$C_0$-semigroup $\uot$ to $X_1$ we employ a strategy inspired by
\cite{positivity} based upon Propositions \ref{muspos} and
\ref{propvoigt}. First, one gets the following invariance result as
a straightforward application of Prop. \ref{propvoigt}:
\begin{theo}\label{invariance} The spectra of $(\u_1(t))_{t \geq 0}$
and $\T_1$ in $X_1$ enjoy the following invariance properties:
$$\sigma(\T_1)=\sigma(\T_1) + i \mathbb{R} \qquad \text{ and }\qquad
\sigma(\u_1(t))=\sigma(\u_1(t))\cdot \mathbb{T} \qquad (t \geq 0)$$
where $\mathbb{T}$ is the unit circle of the complex
plane.\end{theo}
\begin{proof} The proof consists in applying the above Proposition
\ref{propvoigt} to the subset $\O_1$ and the subspace $X_1$ using
the function $$\alpha(\x)=-\tau_+(\x),\qquad \x \in \O_1.$$ One
checks immediately that $\alpha(\cdot)$ fulfills \eqref{alpha}. Note
that $\alpha(\cdot)$ is finite over $\O_1$.\end{proof}

According to the above Theorem, to describe the spectral features of
$(\u_1(t))_{t \geq 0}$, it is sufficient to describe \textit{its
real spectrum}.  We shall denote
$$\Sigma_p(\x)=h(\x)+\frac{1}{p^\star}\mathrm{div}(\ff)(\x), \qquad \x
\in \O,$$ where $p^\star$ is the conjugate of the exponent $1 \leq p
< \infty$ we fixed at the beginning,
$\frac{1}{p^\star}+\frac{1}{p}=1$.
\begin{theo}\label{Spectu1} Assume $\mathfrak{m}(\O_1) \neq 0.$ One has
$$\sigma(\u_1(t))=\sigma_{\mathrm{ap}}(\u_1(t))=\{\xi
\in \mathbb{C}\,;\,|\xi| \leq \exp(-\gamma_1 t)\} \qquad \qquad (t
\geq 0)$$ and
$$\sigma(\T_1)=\sigma_{\mathrm{ap}}(\T_1)=\{\lambda\,;\,\mathrm{Re}\lambda
\leq -\gamma_1\}$$ where
$$\gamma_1=\lim_{t \to \infty}
\inf\left\{t^{-1}\int_0^t \Sigma_p(\Phi(\x,-s))\d s\,;\,\x \in
\O_1,\:t < \tau_-(\x)\right\}.$$
\end{theo}
\begin{proof} Let us denote by $\omega_0(\u_1)$ the type of $(\u_1(t))_{t \geq 0}$ and let us fix $t \geq 0$ and $f \in
X_1$. One has
$$\|\u_1(t)f\|_{X_1}^p=\int_{\{\tau_+(\x) < \infty\}}\exp{\left[-p\int_0^t \nu(\Phi(\x,-s))\d s\right]}
\big|f(\Phi(\x,-t))\big|^p\chi_{\{t <
\tau_-(\x)\}}(\x)\d\mathfrak{m}(\x).$$ Recalling that the finiteness
of $\t_+(\x)$ is equivalent to that of $\t_+(\Phi(\x,-t))$ for any
$t< \t_-(\x)$, the change of variable $\x \mapsto \y=\Phi(\x,-t)$
maps $\O_1$ onto itself. Moreover, according to Proposition
\ref{mea}, the Jacobian of the transformation is given by
$$G(t,\y)=\exp\left[\int_0^t \mathrm{div}(\ff)(\Phi(\y,s))\d
s\right].$$ Now, using that $\Phi(\y, s)=\Theta(\x,t,s)=\Phi(\x,s-t)
\in \O$ as long as $0<t-s<\tau_-(\x),$ one sees that
$$\chi_{\{\tau_+(\x) < \infty\}\cap \{t<\tau_-(\x)\}}(\x)=\chi_{\{\tau_+(\y) < \infty\} \cap \{\tau_+(\y) > t\}}(\y).$$
Therefore,
\begin{equation}\begin{split}\label{normX1}\|\u_1(t)f\|_{X_1}^p&=\int_{\{\t_+(\y) < \infty\}
}\exp{\left[-p\int_0^t \nu(\Phi(\y, r))\d
r\right]}|f(\y)|^p\chi_{\{t <
\tau_+(\y)\}}(\y)G(t,\y)\d\mathfrak{m}(\y)\\
&=\int_{\{\t_+(\y) < \infty\} }\exp{\left[-p\int_0^t
\Sigma_p(\Phi(\y, r))\d r\right]}|f(\y)|^p\chi_{\{t <
\tau_+(\y)\}}(\y)\d\mathfrak{m}(\y).\end{split}\end{equation}
Consequently,
$$\|\u_1(t)\|_{\mathscr{B}(X_1)} =\sup\left\{\exp{\left[-\int_0^t \Sigma_p(\Phi(\y,r))\d
r\right]}\,;\,\t_+(\y) < \infty\,,\,t < \tau_+(\y)\right\}$$ or,
equivalently,
$$\log \|\u_1(t)\|_{\mathscr{B}(X_1)}=-\inf\left\{\int_0^t \Sigma_p(\Phi(\y ,r))\d
r\,;\,\t_+(\y) < \infty\,,\,t < \tau_+(\y)\right\}.$$ Performing now
the converse change of variable $\y \mapsto \x=\Phi(\y, t)$ one sees
as previously that
$$\log \|\u_1(t)\|_{\mathscr{B}(X_1)}=-\inf\left\{\int_0^t \Sigma_p(\Phi(\x,-s ))\d
s\,;\,\t_+(\x) < \infty\,,\,t < \tau_-(\x)\right\}.$$ Finally,
recalling that $\omega_0(\u_1)=\lim_{t\to \infty}t^{-1}\log
\|\u_1(t)\|_{\mathscr{B}(X_1)}$, one deduces that
$\omega_0(\u_1)=-\gamma_1$. Moreover, from the positiveness of
$(\u_1(t))_{t \geq 0}$ and \cite{weis}, one gets $s(\T_1)=-\gamma_1$
where $s(\T_1)$ is the spectral bound of $\T_1$. Let us now show
that the set $\mathfrak{Y}_1$ of local quasi-nilpotence of
$(\u_1(t))_{t \geq 0}$ is dense in $X_1$. Indeed, set $\O_1^m=\{\x
\in \O\,;\,\t_+(\x) \leq m\}$, $m \in \mathbb{N}$. One has
$\O_1=\cup_{m\geq 1}\O_1^m$. Set $Y_1=\cup_{m\geq 1}X_1^m$ where
$X_1^m=\left\{f \in X_1\,;\,f(\x)=0 \: \text{ a.e. } \x \in \O_1
\setminus \O_1^m\right\}.$ One checks easily that $Y_1$ is a dense
sublattice of $X_1$. Now, let $f \in Y_1$. There exists some integer
$m \geq 1$ such that $f(\x)=0$ for almost every $\x \in \O_1$ with
$\t_+(\x) > m.$ From \eqref{normX1}, one gets
$$\|\u_1(t)(|f|)\|_{X_1}=0\qquad \forall t >m$$
which clearly shows that $Y_1$ is a subset of $\mathfrak{Y}_1$ and
 that $\mathfrak{Y}_1$ is dense in $X_1$. From Proposition \ref{muspos}, one gets then that
$$(-\infty,-\gamma_1] \subset \sigma_{\mathrm{ap}}(\T_1)
\qquad \text{ and } \qquad [0,\exp(-\gamma_1 t)] \subset
\sigma_{\mathrm{ap}}(\u_1(t)) \qquad (t \geq 0).$$ Now, the result
follows from the invariance of $\sigma(\T_1)$ and $\sigma(\u_1(t))$
under vertical translations along the imaginary axis and under
rotations respectively (Theorem \ref{invariance}).
\end{proof}

\subsection{Spectral properties of $(\u_2(t))_{t \geq
0}$}\label{subsec:u2}

Let us investigate now the spectral properties of the restriction of
$\uot$ to the subspace $X_2$. As in the previous case, the main
ingredient of our proof is the following special version of
Proposition \ref{propvoigt}.
\begin{theo}\label{invariance2} The spectra of $(\u_2(t))_{t \geq 0}$
and $\T_2$ in $X_2$ enjoy the following invariance properties:
$$\sigma(\T_2)=\sigma(\T_2) + i \mathbb{R} \qquad \text{ and }\qquad
\sigma(\u_2(t))=\sigma(\u_2(t))\cdot \mathbb{T} \qquad (t \geq 0)$$
where $\mathbb{T}$ is the unit circle of the complex
plane.\end{theo}
\begin{proof} The proof consists in applying the above Proposition
\ref{propvoigt} to $(\u_2(t))_{t \geq 0}$ by choosing
$$\alpha(\x)=\tau_-(\x),\qquad \x \in \O_2.$$ One checks
immediately that $\alpha(\cdot)$ fulfills \eqref{alpha} and that
$\alpha(\cdot)$ is finite $\O_2$.\end{proof}

This leads to the following result whose proof is a technical
generalization of \cite{positivity} that we repeat here for the
self-consistency of the paper.
\begin{theo}\label{Spectu2}
Assume $\mathfrak{m}(\O_2)\neq 0$. One has
$$\sigma(\u_2(t))=\{\xi
\in \mathbb{C}\,;\,|\xi| \leq \exp(-\gamma_2 t)\}\qquad \text{ and }
\qquad \sigma(\T_2)=\{\lambda\,;\,\mathrm{Re}\lambda \leq
-\gamma_2\}$$ where
$$\gamma_2=\lim_{t \to \infty}
\inf\left\{t^{-1}\int_0^t \Sigma_p(\Phi(\x,s))\d s\,;\,\x \in
\O_2,\:t < \tau_-(\x)\right\}.$$
\end{theo}
\begin{proof}
The proof relies on duality arguments. We only prove it in the $L^1$
case, the case $1 < p < \infty$ being simpler (see
\cite{positivity}). Since the dual semigroup $(\u_2^\star(t))_{t
\geq 0}$ of $(\u_2(t))_{t \geq 0}$ is not strongly continuous on
$L^{\infty}(\O_2),$ we have to introduce the space of  strong
continuity
$$X_2^{\odot}=\left\{g \in L^{\infty}(\O_2)\,;\,\lim_{t \to 0^+}\sup_{\x \in
\O_2}|\u_2^{\star}(t)g(\x)-g(\x)|=0\,\right\}.$$ The general theory
of $C_0$-semigroups \cite[Chapter 2, Section 2.6]{engel} ensures
that $X_2^{\odot}$ is a closed subset of $L^{\infty}(\O_2)$
invariant under $(\u_2^{\star}(t))_{\t \geq 0}$. For any $t \geq 0$,
define $\u_2^{\odot}(t)$ as the restriction of $\u_2^\star(t)$ to
$X_2^\odot$, i.e. $(\u_2^{\odot}(t))_{t \geq 0}$ is the sun--dual
semigroup of $(\u_2(t))_{t \geq 0}.$. Then, the semigroup
$(\u_2^{\odot}(t))_{t \geq 0}$ is a $C_0$-semigroup of $X_2^\odot$
whose generator is denoted by $\T_2^\odot$ with \cite[Chapter 2,
Section 2.6]{engel},
$$\D(\T_2^\odot)=\left\{g \in \D(\T_2^\star)\,;\,\T_2^\star g \in
X_2^\odot\right\} \qquad \text{ and } \qquad\T_2^\odot g=\T_2^\star
g , \qquad \forall g \in \D(\T_2^\odot),$$ where $\T_2^\star$
denotes the dual operator of $\T_2$ whose domain $\D(\T_2^\star)$ is
dense in $X_2^\odot$. Moreover \cite[Chapter 4, Proposition
2.18]{engel},
$$\sigma(\T_2)=\sigma(\T_2^\star)=\sigma(\T_2^\odot)\qquad \text{ and } \qquad
\sigma(\u_2(t))=\sigma(\u_2^\star(t))=\sigma(\u_2^\odot(t))$$ where
the spectra of $\T_2^\star$and $\u_2^\star(t)$ are intended in the
space $L^{\infty}(\O_2)$ while that of $\T_2^\odot$ and
$\u_2^\odot(t)$ stand for spectra in $X_2^\odot.$ Define
$$\widehat{X_2^\odot}=\left\{g \in X_2^\odot\,;\,\lim_{r \to
\infty}\sup_{\t_-(\x) \geq r}|g(\x)|=0\right\}$$ and let us show
that $\widehat{X_2^\odot}$ is invariant under $(\u_2^\odot(t))_{t
\geq 0}$. Indeed, let $g \in \widehat{X_2^\odot}$ and
$f=(\lambda-T_2^\odot)^{-1}g$ where $\lambda >0$ is large enough.
Using again Proposition \ref{mea}, one can check without difficulty
that the dual of $\u_2(t)$ is given by
$$\u_2^\star(t)g(\x)=\exp{\left[-\int_0^t h(\Phi(\x,s))\d s\right]}g(\Phi(\x,t))$$ for any $t \geq 0$, $\x \in \O_2$ and $g \in
L^\infty(\O_2)$ where we recall that $\tau_+(\x)=\infty$ for any $\x
\in \O_2$. Consequently,
$$f(\x)=(\lambda-\T_2^\odot)^{-1}g(\x)=\int_0^{\infty}\exp{\left[-\lambda t -\int_0^t h(\Phi(\x,s))\d
s\right]}g(\Phi(\x,t))\d t.$$ Recalling that
$\tau_-(\Phi(\x,t))=\tau_-(\x)+t$ and setting
$\underline{h}=\esin_{\y \in \O}h(\y)$, one has
$$\sup_{\t_-(x) \geq r}|f(\x)|\leq\int_0^\infty \exp(-(\lambda+\underline{h}) t) \d t\sup_{\tau_-(\y) \geq
r}|g(\y)|, \qquad \qquad \forall r >0
$$ which proves that $f \in \widehat{X_2^\odot}$ and implies
that $\widehat{X_2^\odot}$ is invariant under the action of
$(\u_2^\odot(t))_{t \geq 0}$. Denote $(\widehat{\u}_2^\odot(t))_{t
\geq 0}$ the restriction of $(\u_2^\odot(t))_{t \geq 0}$ to
$\widehat{X_2^\odot}$. Define now $\widehat{Y}_2$ as the set of $g
\in \widehat{X_2^\odot}$ for which there is some $r \geq 0$ such
that $g(\x)=0$ if $\tau_-(\x) \geq r.$ One claims that
$\widehat{Y}_2$ is dense in $\widehat{X_2^\odot}$. Indeed, for any
integer $m \geq 1$, let $\gamma_m(\cdot)$ be a continuous function
from $[0,\infty)$ to $[0,1]$ such that $\gamma_m(s)=1$ for any $0
\leq s \leq m$ and $\gamma_m(s)=0$ for any $s \geq 2m.$ For any $g
\in \widehat{X_2^\odot}$, one can define
$g_m(\x)=\gamma_m(\tau_-(\x))\,g(\x)$ and prove without major
difficulties that $(g_m)_m \subset \widehat{Y}_2$ and
$\|g_m-g\|_{L^{\infty}(\O_2)} \to 0$ as $m \to \infty.$ This proves
our claim. Now, the set $\widehat{\mathfrak{Y}_2}$ of local
quasi--nilpotence of $(\widehat{\u}_2^\odot(t))_{t \geq 0}$
obviously contains $\widehat{Y_2}$ and is therefore dense in
$\widehat{X_2^\odot}$. According to Proposition \ref{muspos}, one
gets then
$$[0,\exp(\omega_0(\widehat{\u}_2^\odot)\,t)] \subset \sigma_{\mathrm{ap}}(\widehat{\u}_2^\odot(t))\cap \mathbb{R}
\qquad \text{ and } \qquad (-\infty,\omega_0(\widehat{\u}_2^\odot)]
\subset \sigma_{\mathrm{ap}}(\widehat{\T}_2^\odot)\cap \mathbb{R}$$
where $\widehat{\T}_2^\odot$ is the restriction of $\T_2^\odot$ to
$\widehat{X_2^\odot}$ and where we used the identity between the
type $\omega_0(\widehat{\u}_2^\odot)$ of
$(\widehat{\u}_2^\odot(t))_{t \geq 0}$ and the spectral bound of
$\widehat{\T}_2^\odot$. Since $\widehat{X_2^\odot} \subset
X_2^\odot$, one obtains the following inclusion
$$[0,\exp(\omega_0(\widehat{\u}_2^\odot)\,t)] \subset \sigma_{\mathrm{ap}}({\u}_2^\odot(t))\cap \mathbb{R}
\qquad \text{ and } \qquad (-\infty,\omega_0({\u}_2^\odot)] \subset
\sigma_{\mathrm{ap}}({\T}_2^\odot)\cap \mathbb{R}$$ and Theorem
\ref{invariance2} implies that
$$\sigma(\u_2(t))=\sigma(\u_2^\odot(t))=\sigma_{\mathrm{ap}}(\u_2^\odot(t))=\{\xi
\in \mathbb{C}\,;\,|\xi| \leq
\exp(\omega_0(\widehat{\u}_2^\odot)\,t)\}$$ and
$$\sigma(\T_2)=\sigma(\T_2^\odot)=\sigma_{\mathrm{ap}}(\T_2^\odot)=\{\lambda\,;\,\mathrm{Re}\lambda
\leq \omega_0(\widehat{\u}_2^\odot)\}.$$ To conclude the proof, it
remains only to show that
$\omega_0(\widehat{\u}_2^\odot)=-\gamma_2.$ To do so, one notes
easily that
$$\|\widehat{\u}_2^\odot(t)\|_{\mathscr{B}(\widehat{X_2^\odot})}\leq
\omega(t):= \sup \left\{\exp{\left[-\int_0^t h(\Phi(\x,s))\d
s\right]\,;\,\x \in \O_2,\;t < \t_-(\x)} \right\}.$$ Moreover, with
the above notations, let $\psi_m(\x)=\gamma_m(\tau_-(\x))$, $m \geq
1$. One has $\psi_m \in \widehat{X_2^\odot}$ and $\psi_m=1$ on the
set $\{\x \in \O_2\,;\,\t_-(\x) \leq m\}.$ Therefore,
\begin{equation*}\begin{split}
\|\widehat{\u}_2^\odot(t)\|_{\mathscr{B}(\widehat{X_2^\odot})}&\geq
\|\widehat{\u}_2^\odot(t)\psi_m\|_{\widehat{X_2^\odot}}\\
&\geq \sup \left\{\exp{\left[-\int_0^t h(\Phi(\x,s))\d
s\right]\,;\,\x \in \O_2,\;t < \t_-(\x) \leq m}
\right\}.\end{split}\end{equation*} Letting $m \to \infty$, one gets
$\|\widehat{\u}_2^\odot(t)\|_{\mathscr{B}(\widehat{X_2^\odot})}=\omega(t)$.
One obtains immediately that
$\omega_0(\widehat{\u}_2^\odot)=-\gamma_2.$
\end{proof}

The above results give a complete picture of the spectra of $\uot$
and $\T$ when $\mathfrak{m}(\O_3)=0$:

\begin{theo}\label{theoo3vide} Assume that $\mathfrak{m}(\O_3)=0$. Then,
$$\sigma(\u(t))=\{\xi
\in \mathbb{C}\,;\,|\xi| \leq \exp(-\gamma  t)\}\qquad \text{ and }
\qquad \sigma(\T)=\{\lambda\,;\,\mathrm{Re}\lambda \leq -\gamma\}$$
where $\gamma=\min\{\gamma_1,\gamma_2\}$,  $\gamma_1,\gamma_2$ being
given by Theorems \ref{Spectu1} and \ref{Spectu2}. In particular,
$\uot$ satisfies a \textbf{Spectral Mapping Theorem} $\sigma(\u(t))
\setminus \{0\}=\exp\left(t \sigma(\T)\right)$ for any $t \geq 0.$
 \end{theo}
We shall see in the following section that the picture is
drastically different when $\mathfrak{m}(\O_3)\neq 0$. First, we
shall illustrate the above results by several examples.

\subsection{Examples}
The above result is of particular relevance for applications when
$\O$
 is bounded
in some directions.

\begin{exa}
 Let us consider the case of the classical Vlasov
equation described  in \eqref{truevlasov} with a Lorentz force.
Namely, let $\O= \mathcal{ D } \times \mathbb{R}^3$ where
$\mathcal{D}$ is a smooth open subset of $\mathbb{R}^3$. For any
$\x=(x,v) \in \mathcal{D} \times \mathbb{R}^3$, let
$\ff(\x)=(v,\F(v)) \in \mathbb{R}^6$ where the force field
$\mathbf{F}$ is given by the {\it Lorentz force}:
$$\mathbf{F}=\mathbf{F}(v)=q(\mathcal{E} + v \times \mathcal{B})$$
where $\mathcal{E} \in \R^3$ stands for some given  electric field,
$\mathcal{B}\in \R^3$ denotes a given magnetic field and $q$ is the
electric charge of the particle. One assumes in this example that
$\mathcal{E}$ and $\mathcal{B}$ are two \textbf{\textit{constant}}
fields such that $\langle \mathcal{E},\mathcal{B}\rangle \neq 0$ and
that $\mathcal{D}$ is bounded in the $\mathcal{B}$-direction, i.e.
$$\sup\{|\langle x, \mathcal{B}\rangle|, x \in \mathcal{D}\} < \infty,$$
then $\mathfrak{m}(\O_3)= 0$. Indeed, one sees easily that the
solution $(x(t),v(t))$ to the characteristic system
\begin{equation}\label{lorentz}\begin{cases}
\dot{x}(t)&=v(t),\\
\dot{v}(t)&=q(\mathcal{E} + v(t) \times \mathcal{B}),\qquad t \in \mathbb{R}\\
\end{cases}
\end{equation}
with initial condition $x(0)=x,$ $v(0)=v$, is  such that $$\langle
x(t),\mathcal{B}\rangle = \frac{q}{2} \langle
\mathcal{E},\mathcal{B}\rangle t^2  + \langle v,\mathcal{B}\rangle t
+ \langle x, \mathcal{B}\rangle.$$ Since the right-hand side is
unbounded as $t \to \pm \infty$ for any $(x,v) \in \O$, necessarily
$x(t)$ escapes $\mathcal{D}$ in finite time and $\O_3=\varnothing$.
\end{exa}

\begin{nb}
Note that, in general, the solution $(x(t),v(t))$ to the
characteristic system \eqref{lorentz} describes an helix with axis
directed along $\mathcal{B}$ and radius proportional to
$1/|q||\mathcal{B}|$ (Larmor radius).
\end{nb}

\begin{exa} Consider now the following one dimensional relativistic transport equation
\begin{multline}\label{nordvlasov}
\partial_t f(x,p,t) + v(p) \partial_x f(x,v,t)
-\left( pv(p) + (1+p^2)^{-1/2} )\right) \phi'(x)\partial_p
f(x,p,t)\\+ \nu(x,p)f(x,p,t)=0, \qquad (x,v) \in (0,1) \times
\mathbb{R}
\end{multline}
where the potential $\phi$ is a given smooth function, say $\phi \in
W^{2,\infty}(0,1)$, and the relativistic velocity $v(p)$
corresponding to the impulsion $p \in \mathbb{R}$ is given by
$v(p)=\frac{p}{\sqrt{1+p^2}}$.

The study of the above equation \eqref{nordvlasov} is related to the
relativistic Nordstr\"{o}m-Vlasov systems  for plasma (see
\cite{bostan1,bostan2} and the references therein). We can then
define a smooth vector field over $\O=(0,1) \times \mathbb{R}$ by
$$\ff(x,p)=\left(v(p),-\left( pv(p) + (1+p^2)^{-1/2} )\right)
\phi'(x)\right)=\left(\dfrac{p}{\sqrt{1+p^2}},-\sqrt{1+p^2}\phi'(x)\right)$$
It can be proved  \cite[Corollary 2]{bostan1}   that, whenever
$\phi$ is \textit{convex}, the exit time associated to the above
field is  finite for all characteristic curves with non zero
impulsion, i.e., with the notations of the above section,
$$\tau_{\pm}(x,p) < \infty, \qquad \forall (x,p) \in (0,1) \times
(\mathbb{R}\setminus \{0\}).$$ Therefore, Theorem \ref{theoo3vide}
applies to such a problem.\end{exa}

\section{Spectral properties of $(\u_3(t))_{t \in \mathbb{R}}$}

We are dealing in this section with the spectrum of the $C_0$-group
$(\u_3(t))_{t \in \mathbb{R}}$ and its generator $\T_3$. We assume
therefore for this section that $\mathfrak{m}(\O_3)\neq 0$. For any
$\x \in \O_3$, the mapping $s \mapsto \Phi(\x,s)$ is defined over
$\mathbb{R}$. For simplicity, for any $t \in \mathbb{R}$, we will
denote by $\varphi_t$ the mapping
$$\varphi_t\::\:\x \in \O_3 \mapsto \varphi_t(\x)=\Phi(\x,-t).$$
With this notation, the group $(\u_3(t))_{t \in \mathbb{R}}$ is
given by
$$\u_3(t)f( \x)=\exp\left[-\int_0^t\nu(\varphi_{s}(\x))\d
s\right]f(\varphi_t(\x)),\qquad \forall \x \in \O_3,\;\,t \in
\mathbb{R}.$$ Notice that, arguing as above, it is possible to
compute explicitly the type of both the semigroups $(\u_3(t))_{t
\geq 0}$ and $(\u_3(-t))_{t \geq 0}$ (see \cite[Theorem 1]{zhang}).
As we shall see further on, the arguments of the previous section do
not apply in general to the group $(\u_3(t))_{t \in \R}$. However,
for very peculiar cases, it is possible to proceed in the same way.
More precisely, we recall the following definition of section, taken
from \cite{bathia}.
\begin{defi}\label{defi:bathia} A set $\mathcal{S} \subset \O_3$ is said to be
a section of $\O_3$ (associated to the flow $\varphi$) if, for any
$\x \in \O_3$, there exists a unique  real number $\theta(\x)$ such
that $\varphi_{\theta(\x)}(\x) \in \mathcal{S}.$ \end{defi} One can
state the following:
\begin{theo}\label{theo:sect} If there exists a\textbf{ measurable} section $\mathcal{S}$
associated to the flow $\varphi$, then
$$\sigma(\T_3)=\sigma (\T_3)+i\mathbb{R} \qquad \text{
and } \qquad\sigma (\u_3(t))=\sigma (\u_3(t))\cdot \mathbb{T},
\qquad (t \in \mathbb{R}).$$ As a consequence,
\begin{equation*}\label{SMTdis}
\sigma (\u_3(t))=\exp\left(t \sigma (\T_3)\right), \qquad (t \in
\mathbb{R}).\end{equation*}
\end{theo}
\begin{proof} As in the previous section, the proof consists in
exhibiting a function $\alpha(\cdot)$ satisfying \eqref{alpha}.
Precisely, set
$$\alpha(\x)=-\theta(\x),\qquad \x \in \O_3$$
where $\theta(\cdot)$ is provided by Definition \ref{defi:bathia}.
Since the section $\mathcal{S}$ is measurable, it is easy to see
that $\alpha(\cdot)$ is measurable  and is finite almost everywhere.
Moreover, from the definition of $\theta(\x)$ as the unique 'time'
at which a trajectory starting from $\x$ meets $\mathcal{S}$, one
has
$$\alpha(\varphi_t(\x))=t+\alpha(\x)$$
and we conclude the proof thanks to Proposition \ref{propvoigt}.
\end{proof}

\begin{nb} Notice that  there exists a section $\mathcal{S}$ associated to the
flow $(\varphi_t)_{t \in \mathbb{R}}$ if and only if $(\varphi_t)_{t
\in \mathbb{R}}$ does not admit periodic orbits nor rest points \cite[p. 48]{bathia}.
However, it is a difficult task to provide sufficient conditions on
the flow $(\varphi_t)_{t \in \mathbb{R}}$ ensuring the section
$\mathcal{S}$ to be measurable.  Let us however mention that,
whenever the flow $(\varphi_t)_{t \in \mathbb{R}}$ is dispersive
over $\O_3$ in the sense of \cite[Chapter IV]{bathia} then
$\mathcal{S}$ is measurable (and the function $\theta(\cdot)$
provided by Definition \ref{defi:bathia} is continuous).\end{nb}

\begin{exa}\label{neutron} Adopting notations from neutron transport theory, assume $\O=\mathbb{R}^N \times
V$ where $V$ is a closed subset of $\mathbb{R}^N$. Then, for any
$\x=(x,v) \in \O$, with $x \in \mathbb{R}^N$ and $v \in V$, define
$\ff(\x)=(v,0)$. Then, the mapping $\alpha(\x)=(x \cdot v)/|v|^2$
for any $v \neq 0$, fulfills \eqref{propvoigt} and the spectrum of
$(\u(t))_{t \in \mathbb{R}}$ and $\T$ are invariant by rotations and by vertical
translations along the imaginary axis respectively (see
\cite{positivity,voigt} for more details).
\end{exa}

Apart from the very special result  above, one has to provide an
alternative approach to deal with the spectral properties of
$(\u_3(t))_{t \in \mathbb{R}}$.
%
%
The aim of this section is to get a more precise picture of the
spectrum of both $(\u_3(t))_{t \in \R}$ and $\T_3$. Recall now that,
for a point $\x \in\O_3$, since $\ff$ is globally Lipschitz, three
situations may occur:
\begin{enumerate}
\item $\x$ is a rest point of the flow $(\varphi_t)_t$, i.e.
$\ff(\x)=0$;
\item $\x$ belongs to a periodic orbit of $(\varphi_t)_t$, i.e.
there is some $t_0 >0$ such that $\varphi_{t_0}(\x)=\x.$
\item $\x$ belongs to an  infinite but non closed orbit.
\end{enumerate}
This leads to the following splitting of $\O_3$:
$$\O_3=\O_{\mathrm{rest}} \cup \O_{\mathrm{per}}\cup  \O_{\infty}$$
where
$$\O_{\mathrm{rest}}=\{\x \in \O_3\,;\,\varphi_t(\x)=\x \;\:\forall t \in
\mathbb{R}\}, \qquad \O_{\mathrm{per}}=\{\x \in \O_3\,:\,\exists t >
0, \varphi_t(\x)=\x\}$$ and
$$\O_\infty=\O_3 \setminus (\O_{\mathrm{rest}} \cup \O_{\mathrm{per}}).$$
Notice that $\O_\mathrm{rest}$ is clearly a closed subset of $\O$
while $\O_\mathrm{per}$ is measurable (this follows from the fact
that the sets $\O_{\mathrm{per},\k}$ defined by \eqref{opern} are
closed for any $\k \in \mathbb{N}$ according to \cite[p.
314]{arendtgreiner}). It is not difficult to see that these sets are
all invariant under the flow $(\varphi_t)_{t \in \mathbb{R}}$ and
consequently, under the action of $\u_3(t)$ for any $t \in
\mathbb{R}$. Therefore, defining $X_{\mathrm{rest}}$,
$X_{\mathrm{per}}$ and $X_\infty$ as the set of functions
  in $X_3$ which are null almost everywhere outside of the sets $\O_{\mathrm{rest}}$,
  $\O_{\mathrm{per}}$ and $\O_{\infty}$ respectively, one can define the
  following restrictions of $\u_3(t)$:
$$\u_{\mathrm{rest}}(t)=\u_3(t)_{|X_{\mathrm{rest}}},\quad
\u_{\mathrm{per}}(t)=\u_3(t)_{|X_{\mathrm{per}}} \quad \text{and}
\quad \u_{\infty}(t)=\u_3(t)_{|X_\infty}.$$ Clearly,
$(\u_{\mathrm{rest}}(t))_{t \in \mathbb{R}}$,
$(\u_{\mathrm{rest}}(t))_{t \in \mathbb{R}}$ and
$(\u_{\infty}(t))_{t \in
  \mathbb{R}}$ are positive $C_0$-groups of $X_{\mathrm{rest}}$,
$X_{\mathrm{rest}}$ and $X_\infty$ respectively, whose generators will be
  denoted respectively by $\T_{\mathrm{rest}}$, $\T_{\mathrm{per}}$ and $\T_\infty.$
Arguing as in Theorem \ref{decomp}, the spectra of $\u_3(t)$ and $\T_3$ are
given respectively by
\begin{equation}\begin{split}\label{spect:decomp3}
\sigma(\u_3(t))&=\sigma(\u_{\mathrm{rest}}(t)) \cup
\sigma(\u_{\mathrm{per}}(t)) \cup \sigma(\u_\infty(t)),\\
\sigma(\T_3)&=\sigma(\T_{\mathrm{rest}}) \cup
\sigma(\T_{\mathrm{per}}) \cup
\sigma(\T_\infty).\end{split}\end{equation}

 Because of the
possible existence of periodic orbits, it is not clear {\it a
priori} that a result analogous to Theorems \ref{invariance} and
\ref{invariance2} can be proved in this case (even if
$\nu(\cdot)=0$). Indeed, if $\mathfrak{m}(\O_{\mathrm{per}}) \neq
0$, then there is no mapping $\alpha(\cdot)$ satisfying
\eqref{alpha} on $\O_3.$ Indeed, if such a mapping were exist, in
particular, for any $\x \in \O_{\mathrm{per}}$,
$$\alpha(\varphi_t(\x))=\alpha(\x)+t, \qquad \forall t \geq 0.$$
However, there exists a period $t_0\neq 0$ such that
$\varphi_{t_0}(\x)=\x$ so that
$t_0=\alpha(\varphi_{t_0}(\x))-\alpha(\x)=0$ which is a
contradiction. Of course, the impossibility to construct a function
$\alpha(\cdot)$ with the above properties does not imply that the
spectrum of $\T_3$ is not invariant by vertical translations along
the imaginary axis.

\begin{exa}\label{exa:cercle}  Let us consider the planar field $\ff(\x)=(-y,x)$ for any
$\x=(x,y) \in \O=\R^2.$ In such a case, the characteristic curves
are \textit{circular}, namely
$$\Phi(\x,s)=(x\cos s-y\sin s, x\sin s + y \cos s), \qquad
\x=(x,y),\qquad s \in \mathbb{R}.$$ In particular, for any $\x=(x,y)
\in \O$ one has $\tau_{\pm}(\x)=\infty$ and the mapping:
$$t \in \mathbb{R} \longmapsto \varphi_t(\x)=(x\cos t + y\sin t, y \cos t-x \sin t)$$
is $2\pi$--periodic. This means that $\O_1$ and $\O_2$ are both
empty sets.  Consider for a while $\nu(\cdot)=0$. The semigroup
$(\u(t))_{t \geq 0}$ extends to a group and one has
$$\u(t)f(\x)=f(\varphi_t(\x)), \qquad \forall t \in \mathbb{R},\:f \in X.$$
Clearly, $(\u(t))_{t \in \mathbb{R}}$ is a positive and periodic
$C_0$--group of $X$ with period $2\pi$. As a consequence
\cite{nagel86}, $\sigma(\T) =i\mathbb{Z}.$  In particular,
$\sigma(\T) \neq \sigma(\T)+i\mathbb{R}.$
\end{exa}

The previous example shows that one cannot hope to deduce the
spectral properties of $\T_3$ or $(\u_3(t))_{t \in \mathbb{R}}$ from
their respective real spectrum as it was the case for the
restrictions of $\uot$ to $X_1$ and $X_2$.

\subsection{Stationary flow}

We first deal with the spectral structure of the part
$(\u_{\mathrm{rest}}(t))_{t \in \mathbb{R}}$ of the $C_0$-group
$(\u(t))_{t \in \mathbb{R}}$ acting on the set of rest points
$\O_{\mathrm{rest}}$. Since $\varphi_t(\x)=\x$ for any $\x \in
\O_{\mathrm{rest}}$ and any $t \in \mathbb{R}$, it is clear that
$\u_{\mathrm{rest}}(t)$ is given by:
$$\u_{\mathrm{rest}}(t) f(\x)=\exp(-t \nu(\x))f(\x),\qquad \x \in
\O_{\mathrm{rest}},\:f \in X_{\mathrm{rest}}.$$ Hence,
$(\u_{\mathrm{rest}}(t))_{t \in \mathbb{R}}$ is a positive
$C_0$-group of multiplication. This leads naturally to the following
description of  its generator:
$$\T_{\mathrm{rest}}f=-\nu(\cdot)f,\qquad \qquad f
\in \D(\T_{\mathrm{rest}})$$ which can easily be deduced from the
definition of $(\u_{\mathrm{rest}}(t))_{t \in \mathbb{R}}$. Since
the spectrum of $\T_{\mathrm{rest}}$ is real, one can deduce from
the spectral mapping theorem for the real spectrum (see Theorem A.
\ref{realspec}) the following:
\begin{theo}\label{theo:urest} Assume that
  $\mathfrak{m}(\O_{\mathrm{rest}})\neq 0.$
The spectra of $(\u_{\mathrm{rest}}(t))_{t \in \mathbb{R}}$ and its
generator $\T_{\mathrm{rest}}$ satisfy the following {\it spectral
mapping theorem}:
$$\sigma(\u_{\mathrm{rest}}(t))={\exp(t \sigma(\T_{\mathrm{rest}}))},\qquad
\qquad \forall t \in \mathbb{R}$$ where the spectrum of the
generator is given by
$$\sigma(\T_{\mathrm{rest}})=\mathcal{R}_{\mathrm{ess}}(-\nu(\cdot)) \subset
\mathbb{R}$$ where $\mathcal{R}_{\mathrm{ess}}(-\nu(\cdot))$ denotes the essential range of
$-\nu(\cdot)$.
\end{theo}
\begin{nb}
Note that the spectrum of $\T_{\mathrm{rest}}$ is not necessarily
connected.
\end{nb}

\subsection{Aperiodic flow}

The study of the spectral properties of the $C_0$-group associated
to an aperiodic flow has been investigated in a framework of
continuous functions in \cite{arendtgreiner} and in the more general
framework of Mather semigroups in \cite{chicone} under the
supplementary assumption: \begin{hyp}\label{hypcontinu} The mapping
$\x \in \O_\infty \longmapsto \ds\int_0^t h(\varphi_s(\x))\d s$  is
\textit{\textbf{continuous}} for any  $t \in \R.$
\end{hyp}
Namely, one can prove the
following which is a consequence of \cite[Theorem 6.37,
p.188]{chicone}:
\begin{theo}\label{theoAperiod} Assume that $\mathfrak{m}(\O_{\infty})\neq 0$ and Assumption \ref{hypcontinu} holds true. Then,
$$\sigma(\T_\infty)=\sigma(\T_\infty)+i\,\mathbb{R} \qquad \text{
and } \qquad\sigma(\u_\infty(t))=\sigma(\u_\infty(t))\cdot
\mathds{T}, \qquad (t \in \mathbb{R}).$$ As a consequence,
\begin{equation}\label{SMT}
\sigma(\u_\infty(t))=\exp\left(t \sigma(\T_\infty)\right), \qquad (t
\in \mathbb{R}).\end{equation}
\end{theo}


\begin{proof} Using the terminology of \cite{chicone}, the group $(\u_\infty(t))_{t \in \R}$ is
an evolution Mather group induced by a cocycle
$\left(\mathbf{\Psi}_t\right)_t$ over a flow
$\left(\psi_t\right)_t$. Indeed,
$$\u_\infty(t)f(\x)=J_t(\x)\mathbf{\Psi}_t(\psi_{-t}(x))\,f(\psi_{-t}(x)),\qquad
t \in \mathbb{R},\;\x \in \O_\infty,\,\,f \in X_\infty$$ where the
flow $(\psi_t)_t$ is given by $\psi_{-t}=\varphi_t$ while
\begin{equation*}\begin{cases}
\mathbf{\Psi}\;\;:\:\:&\O_\infty \times \mathbb{R} \to \mathbb{R}\\
&(\x,t)\mapsto \mathbf{\Psi}_t(\x)=\exp\left(-\ds \int_0^t
h(\varphi_s(\x))\d s\right)\end{cases}\end{equation*} and
 $J_t(\x)=\exp\left[-\int_0^t \mathrm{div}(\ff)(\varphi_{-s}(\x))\d
s\right]$  is the Radon-Nikodym derivative of the $\d\mathfrak{m}
\circ \psi_t$ with respect to $\mathfrak{m}$ (see Proposition
\ref{mea}). The fact that $\left(\mathbf{\Psi}_t\right)_t$ is a
cocycle of $\O_\infty$ over the flow $\left(\psi_t\right)_t$ is a
consequence of the following straightforward  identities
$$\mathbf{\Psi}_{t+s}(\x)=\mathbf{\Psi}_t(\psi_{-s}(\x))\mathbf{\Psi}_s(\x)
\qquad \text{ and } \qquad \mathbf{\Psi}_0(\x)=1 \qquad \forall \x
\in \O_\infty,\:t,s \in \mathbb{R}.$$ Therefore, one sees that
\cite[Theorem 6.37]{chicone} applies since Assumption
\ref{hypcontinu} implies that the mapping $(\x,t) \in
\O_\infty\times \mathbb{R}\mapsto \mathbf{\Psi}_t(\x) \in
\mathbb{R}$ is continuous.
\end{proof}

\begin{nb} Notice that a more precise picture of the spectrum of
$\sigma(\T_\infty)$ is still missing. We also point out that that the rotational invariance of the spectrum
$\sigma(\u_\infty(t))$ for any $t \in \mathbb{R}$ has been proved in
\cite{kitover} in a direct way.\end{nb}

\begin{nb} Assumption \ref{hypcontinu}  is needed here in order to apply directly the results of \cite[Chapter 6]{chicone}. One may wonder if it is possible to get rid of such an assumption.
\end{nb}

\begin{nb}\label{aperiodic} From the results of the previous sections, one can deduces that,
when the flow $\varphi(\cdot)$ is {\it aperiodic} then the semigroup
$\uot$ satisfies the spectral mapping theorem. Precisely, if
$\mathfrak{m}(\O_\mathrm{per})=0$ and Assumption \ref{hypcontinu}
holds, then
\begin{equation} \label{SpMaTh}
\sigma(\u(t))\setminus \{0\}=\exp(t \sigma(\T)),\qquad \forall t
\geq 0.\end{equation} If moreover,
$\mathfrak{m}(\O_\mathrm{rest})=0$, then
$\sigma(\T)=\sigma(\T)+i\R$.
\end{nb}

Note that practical criteria ensuring $\O_\mathrm{per}$ to be {\it
empty} are well-known. For a $\mathscr{C}^1$-planar field
$\ff=(\ff_1,\ff_2)$ (i.e. when $N=2$), one can mention the so-called
{\it Dulac's criterion} \cite[Proposition 24.14]{amman} which states
that, if $\O$ is simply connected, and if there exists $\varrho \in
\mathscr{C}^1(\O,\mathbb{R})$ such that $\mathrm{div}(\ff \varrho)$
is not identically zero and does not change sign in $\O$, then there
are no periodic orbits lying entirely in $\O$. Generalizations to
higher dimension ($N \geq 3$) can also be provided (see e.g.,
\cite{li,beretta}). For planar $C^1$ field, one has also the
following  useful criterion:
\begin{cor} Let $I$ be an open interval of $\R$ and let $F \in \mathscr{C}^1(I
\times \R\,;\,\R)$. Consider the planar field $\ff(\x)=(v,F(x,v))$
for any $\x=(x,v) \in \O:=I \times \R$. If
$$F(x,0)\neq 0, \qquad \forall x \in I,$$
then $\O_\mathrm{rest}=\O_{\mathrm{per}}=\varnothing$ and the
spectral mapping theorem \eqref{SpMaTh} holds for any $h$ such that
Assumption \ref{hypcontinu} is met.
\end{cor}
\begin{proof} Under the above assumption, $\ff$ does not possess any equilibrium point in $\O$, i.e.
$\O_\mathrm{rest}=\varnothing$. According to \cite[Corollary 24.22,
p. 346]{amman} there is no periodic orbit whose interior lies
completely in $\O=I \times \R$. Since $I$ is an interval of $\R$,
$\O_\mathrm{per}=\varnothing$ and we get the conclusion from Remark \ref{aperiodic}.
\end{proof}

Finally, it is also known that planar  {\it gradient flows} do not exhibit
any periodic orbit \cite[p. 241]{amman} leading to the following
useful corollary
\begin{cor}\label{gradient}
Assume that there exists $V\::\:\R^N \to \mathbb{R}$ such that
 $\ff(\x)=-\nabla V(\x),$ $ \forall \x \in \R^N.$
Then $\sigma(U(t))\setminus \{0\}=\exp(t \sigma(T)),$ for any $t \geq 0$ as soon as
Assumption \ref{hypcontinu} is met.
\end{cor}

We finally illustrate the above results by examples taken from
various kinetic equations.

\subsection{Examples} Let us begin with the classical Vlasov
equation with a quadratic potential
\begin{exa}  One considers, as in
\eqref{vlasov}, a cylindrical domain $\O= {D}\times \mathscr{V}
\subset \mathbb{R}^6$ where $D$ is a smooth open subset of
$\mathbb{R}^3$. For any $\x=(x,v) \in D \times \mathscr{V}$, let
\begin{equation}\label{vlasov1}\ff(\x)=(v,\F(x,v)) \in \mathbb{R}^6\end{equation} where
$\F=(\F_1,\F_2,\F_3)$ is a time independent globally Lipschitz field
(the force field) over $D \times \mathscr{V}$ given by
$$\F(x,v)=-x-\nabla_v \mathbf{U}(v), \qquad \forall (x, v) \in \O$$
where $\mathbf{U}(v)$ is a Lipschitz space homogeneous potential
$\mathbf{U}: \mathscr{V} \to \R.$ Then,  $\ff$ is a gradient flow
associated to the potential $V(x,v)= \langle v, x \rangle +
\mathbf{U}(v).$ Therefore, the associated transport operator and
semigroups are satisfying Corollary \ref{gradient}. 
\end{exa}
\begin{exa} Of course, the classical neutron transport equation for
which $\O=D \times \mathscr{V} \subset \mathbb{R}^6$ and
$$\ff(x,v)=(v,0),\qquad \forall
\x=(x, {v}) \in \O$$ is such that
$\mathfrak{m}(\O_\mathrm{per})=0$.\end{exa}

More surprisingly, the above results also apply to kinetic equations
of second order via a suitable use of Fourier transform
\cite{desvill} :

\begin{exa} Consider as in \cite{desvill} the Vlasov-Fokker-Planck
equation with quadratic confining potential:
\begin{equation}\label{originalVFP}
\partial_t f +v \cdot \nabla_x f - x \cdot \nabla_v f
= \nabla_v \cdot \left(\nabla_v f + v f\right)\end{equation} where
$f=f(x,v,t) \in L^2(\R^N \times \R^N)$ for any $t \geq 0$. The above
equation is unitarily equivalent to the following first order
equation in $\xi$ and $\eta$:
\begin{equation}\label{VFP}
\partial_t \hat{f} +
\eta \cdot \nabla_\xi \hat{f} + (\eta-\xi) \cdot \nabla_\eta \hat{f}
+ |\eta|^2 \hat{f}=0\end{equation} where
$$\hat{f}(\xi,\eta,t)=\int_{\R^N \times \R^N} \exp(-i(x \cdot \xi +
v \cdot \eta)) f(x,v,t)\d v \d x$$ denotes the $L^2$ Fourier
transform (in $x$ and $v$) of $f$. The above equation \eqref{VFP}
falls within the theory we developed in the previous sections.
Precisely, let
$$\ff(\x)=\ff(\xi,\eta)=(\eta,\eta-\xi), \qquad \forall \x=(\xi,\eta) \in \O=\R^N \times \R^N.$$
Notice that $\mathrm{div}\ff(\x)=N$ for any $\x \in \O$ and,
according to \eqref{VFP}, $h(\x)=h(\xi,\eta)=|\eta|^2-N$. The
characteristic system
$$\dot{\xi}(t)=\eta(t),\qquad \qquad \dot{\eta}(t) =(\eta(t)-\xi(t)),\qquad t \in
\mathbb{R}$$ with initial condition $\xi(0)=\xi_0,$ $\eta(0)=\eta_0$
is explicitly solvable \cite{desvill} with
\begin{equation*}\begin{split}
\xi(t)&=\frac{2}{\sqrt{3}}\exp(t/2)\left\{\left(\frac{\sqrt{3}}{2}\cos\left(\frac{\sqrt{3}}{2}t\right)-\frac{1}{2}
\sin\left(\frac{\sqrt{3}}{2}t\right)\right)\xi_0
+\sin\left(\frac{\sqrt{3}}{2}t\right)\eta_0\right\}\\
\eta(t)&=\frac{2}{\sqrt{3}}\exp(t/2)\left\{\left(\frac{\sqrt{3}}{2}\cos\left(\frac{\sqrt{3}}{2}t\right)+\frac{1}{2}
\sin\left(\frac{\sqrt{3}}{2}t\right)\right)\eta_0
-\sin\left(\frac{\sqrt{3}}{2}t\right)\xi_0\right\}.
\end{split}\end{equation*}
The flow $\varphi_t \::(\xi_0,\eta_0) \mapsto (\xi(t),\eta(t))$ does
not possess any periodic orbit while $(0,0)$ is the unique rest
point. Therefore, according to the above Theorem \ref{aperiodic},
the $C_0$-group governing    \eqref{VFP} satisfies the Spectral
Mapping Theorem \ref{SpMaTh}. Turning back to the original
variables, the Vlasov-Fokker-Planck $(\u_{VFP}(t))_{t \in \R}$
governing Eq. \eqref{originalVFP} also satisfies the Spectral
Mapping Theorem \ref{SpMaTh}: $$\sigma(\u_{VFP}(t))=\exp(t
\sigma(L_{VFP})),\qquad \forall t \in \R,$$ where $L_{VFP}$ is the
Vlasov-Fokker-Planck operator:
$$L_{VFP}f(x,v)=-v \cdot \nabla_x f + x \cdot \nabla_v f
 + \nabla_v \cdot \left(\nabla_v f + v f\right),$$ with its maximal
 domain in $L^2(\R^N \times \R^N).$
\end{exa}

\section{Periodic flow}\label{sec:perio}

We now deal with the study of the periodic part of the group
$\u(t)$, by studying the spectral properties of
$(\u_{\mathrm{per}}(t))_{t \in \mathbb{R}}$ on the space
$X_{\mathrm{per}}$. In contrast to what happens for aperiodic flow,
for periodic flow the spectral mapping theorem
\begin{equation}\label{SPM}
\sigma(\u_{\mathrm{per}}(t))=\exp\left(t \sigma
(\T_{\mathrm{per}})\right), \qquad \forall t \in
\mathbb{R}\end{equation} does not hold. Indeed, let us consider the
following example

\begin{exa}\label{exa:cercle1} We turn back to the rotation group in $\R^2$ introduced
in Example \ref{exa:cercle}. Recall that, for such an example
 $\mathbf{\Omega}=\mathbb{R}^2$, $\ff(\x)=(-y,x)$ for any $\x=(x,y)$ and the associated
$C_0$-group in $L^p(\R^2,\d\mathfrak{m})$ is given by
$\u(t)f(\x)=f(\varphi_t(\x))$, $t \in \R$ where the flow $\varphi_t$
is given by the rotation of angle $t \in \R:$
$$\varphi_t(\x)=\left(%
\begin{array}{cc}
  \cos t  & \sin t \\
  -\sin t & \cos t \\
\end{array}%
\right)\left(%
\begin{array}{c}
  x \\
  y \\
\end{array}%
\right), \qquad \qquad \x=(x,y) \in \R^2.$$ According to
\cite{nagel86}, the spectrum of the generator $\T$ of $(\u (t))_{t
\in \mathbb{R}}$ is given by $\sigma(\T)=i\mathbb{Z}$. Moreover, for
any $t \in \R$,  $\mu_n(t)=\exp(int)$ is an eigenvalue of $\u(t)$
for any $n \in \mathbb{Z}$. 
If $t/2\pi$ is irrational, the eigenvalues $\{\mu_n(t),\,n \in
\mathbb{Z}\}$ describe a dense subset of $\mathbb{T}$ and the
closedness of the spectrum of $\u(t)$ implies that
$\sigma(\u(t))=\mathbb{T}$ while $\exp(t
\sigma(\T))=\exp(it\mathbb{Z}) \neq \mathbb{T}$ for such a $t$. This
shows that, in general, the Spectral Mapping Theorem \ref{SPM} fails
for periodic flows.
\end{exa}

For any $\x \in \O_{\mathrm{per}}$, one can define the \textit{prime
period} of $\x$ as
$$\mathfrak{p}(\x)=\inf\{t >0,\,\varphi_t(\x)=\x\}.$$
  The main properties of the prime period are listed in the
following Proposition. We refer the reader to \cite{amman,
arendtgreiner} for the first assertions while the last one is
referred to as Yorke's Theorem \cite{yorke}.
\begin{propo}\label{propyork}  The mapping $\mathfrak{p}(\cdot)$ : $\x \in \O_{\mathrm{per}} \mapsto \mathfrak{p}(\x) \in
(0,\infty)$ enjoys the following properties:
\begin{enumerate}[(i)\:]

\item For any $\x \in \O_{\mathrm{per}}$, $\varphi_t(\x)=\x$ if and
only if $t=n\mathfrak{p}(\x)$ for some $n \in \mathbb{Z}.$
\item $\mathfrak{p}(\cdot)$
is lower semicontinuous, and thus measurable. \item
$\mathfrak{p}(\cdot)$ is invariant under the flow $\varphi_t$, i.e.
$\mathfrak{p}(\varphi_t(\x))=\mathfrak{p}(\x)$ for any $t \in
\mathbb{R}.$\item $\mathfrak{p}(\cdot)$ is bounded away from zero.
Namely, $$\mathfrak{p}(\x) \geq 2\pi /\kappa \qquad \qquad \text{
for any } \quad \x \in \O_{\mathrm{per}},$$ where $\kappa >0$ is the
Lispchitz constant of the field $\ff$.
\end{enumerate}\end{propo}

Notice that, \textit{a priori}, the prime period $\mathfrak{p}$ is
an unbounded function. However, we shall prove that the description
of $\sigma(\T_\mathrm{per})$ relies actually on the behavior of
$\T_\mathrm{per}$ on functions supported on sets where
$\mathfrak{p}$ is bounded. Precisely, for any $ \k
>0$, define $\O_{\mathrm{per},\k}$ as :
\begin{equation}\label{opern}\O_{\mathrm{per},\k}:=\left\{\x \in
\O_\mathrm{per}\,;\,\mathfrak{p}(\x) \leq \k\right\}.\end{equation}
Proposition \ref{propyork} asserts that $\O_{\mathrm{per},\k} =
\varnothing$ for any $0 < \k < 2\pi/\kappa$ and
$\O_{\mathrm{per},\k}$ is invariant under the action of the flow
$(\varphi_t)_t$ for any $\k >0$. As above, one can define
$X_{\mathrm{per},\k}$ as
$$X_{\mathrm{per},\k} =\left\{f \in X_\mathrm{per}\,;\,f(\x)=0
\quad \mathfrak{m}-\text{a.e. } \x \in \O_\mathrm{per} \setminus
\O_{\mathrm{per},\k} \right\}, \qquad \forall \k \geq 2\pi/\kappa,$$
and let $(\u_{\mathrm{per},\k}(t))_{t \in \R}$ be the restriction of
$(\u_\mathrm{per}(t))_{t \in \R}$ on $ X_{\mathrm{per},\k}$. We
denote by $\T_{\mathrm{per},\k}$ its generator. One has the
following abstract result:
\begin{theo}\label{spectn} Let $\lambda \in \mathbb{C}$ be given. Then, $\lambda
\in \varrho(\T_\mathrm{per})$ if and only if
\begin{equation}\label{tpern}
\lambda \in \bigcap_{\k \in
\mathbb{N}}\varrho(\T_{\mathrm{per},\k}), \qquad \text{ and } \qquad
\sup_{\k \in
\mathbb{N}}\left\|\left(\lambda-\T_{\mathrm{per},\k}\right)^{-1}\right\|_{\mathscr{B}(X_{\mathrm{per},\k})}
< \infty.
\end{equation}
\end{theo}
\begin{proof} Assume first that $\lambda \in
\varrho(\T_\mathrm{per})$ and let $\k \in \mathbb{N}$. Given $g \in
X_{\mathrm{per},\k} \subset X_\mathrm{per}$, there is a unique $f
\in X_\mathrm{per}$ such that $(\lambda-\T_\mathrm{per})f=g$. It is
clear that, $f \in X_{\mathrm{per},\k}$ which proves that $\lambda
\in \varrho(\T_{\mathrm{per},\k})$. Moreover,
$\|f\|_{X_{\mathrm{per},\k}} \leq
\|(\lambda-\T_\mathrm{per})^{-1}\|_{\mathscr{B}(X_{\mathrm{per}})}
\|g\|_{X_{\mathrm{per},\k}},$ which proves that
$$\sup_{\k \in
\mathbb{N}}\left\|\left(\lambda-\T_{\mathrm{per},\k}\right)^{-1}\right\|_{\mathscr{B}(X_{\mathrm{per},\k})}
\leq
\left\|\left(\lambda-\T_\mathrm{per}\right)^{-1}\right\|_{\mathscr{B}(X_\mathrm{per})}.$$
Conversely, assume Eq. \eqref{tpern} to holds. Let $g \in
X_\mathrm{per}$. For any $\k \in \mathbb{N}$, let $g_\k=g
\chi_{\O_{\mathrm{per},\k}}$. It is clear that $g_\k \in
X_{\mathrm{per},\k}$. Since $\lambda \in
\varrho(\T_{\mathrm{per},\k})$, there is a unique $f_\k \in
X_{\mathrm{per},\k}$ such that
$(\lambda-\T_{\mathrm{per},\k})f_\k=g_\k$. Now, one notes that, for
any $\x \in \O_{\mathrm{per}}$, $\lim_{\k \to
\infty}\chi_{\O_{\mathrm{per},\k}}(\x)=1$. Consequently, $g_\k$
converges to $g$ in $X_\mathrm{per}$. Let us prove that $(f_\k)_\k$
also converge in $X_\mathrm{per}$. Given $\k_1 \leq \k_2$, one has
$f_{\k_1}-f_{\k_2} \in X_{\mathrm{per},\k_2}$ and
\begin{equation*}\begin{split}\|f_{\k_1}-f_{\k_2}\|_{X_\mathrm{per}}&=\|f_{\k_1}-f_{\k_2}\|_{X_{\mathrm{per},\k_2}}
\leq \sup_{\k \in
\mathbb{N}}\left\|\left(\lambda-\T_{\mathrm{per},\k}\right)^{-1}\right\|_{\mathscr{B}(X_{\mathrm{per},\k})}
\,\|g_{\k_1}-g_{\k_2}\|_{X_{\mathrm{per},\k_2}}\\
&\leq \sup_{\k \in
\mathbb{N}}\left\|\left(\lambda-\T_{\mathrm{per},\k}\right)^{-1}\right\|_{\mathscr{B}(X_{\mathrm{per},\k})}
\,\|g_{\k_1}-g_{\k_2}\|_{X_{\mathrm{per}}} , \qquad \k_1 \leq
\k_2.\end{split}\end{equation*}Therefore, $(f_\k)_\k$ is a Cauchy
sequence in $X_\mathrm{per}$. Let $f$ denote its limit. Since $f_\k
\to f$ and $g_\k \to  g$, one has $\T_{\mathrm{per},\k}f_\k$
converges to $\lambda f-g$. Since $\T_\mathrm{per}$ is a closed
operator, one gets $f \in \D(\T_\mathrm{per})$ with
$(\lambda-\T_\mathrm{per})f=g.$ Now, let $h \in \D(\T_\mathrm{per})$
be another solution to the spectral problem
$(\lambda-\T_\mathrm{per})h=g$. Then, since $\lambda \in
\varrho(\T_{\mathrm{per},\k})$ for any $\k$, one sees that
$h\chi_{\O_{\mathrm{per},\k}}=f_\k$ for any $\k \in \mathbb{N}$.
Using again the fact that $\chi_{\O_{\mathrm{per},\k}}(\x) \to 1$
for any $\x \in \O_\mathrm{per}$, one sees that $h(\x)=f(\x)$ for
any $\x \in \O_\mathrm{per}.$ This proves that $\lambda \in
\varrho(\T_\mathrm{per})$.
\end{proof}

We describe now more precisely the spectrum
$\sigma(\T_{\mathrm{per},n})$. For any $\lambda \in \mathbb{C}$,
define
$$\vartheta(\x)=-\dfrac{1}{\mathfrak{p}(\x)}\int_0^{\mathfrak{p}(\x)}\nu(\varphi_s(\x))\d
s \quad \text{ and } \quad
M_{\lambda}(\x)=\exp\left\{-\mathfrak{p}(\x)(\lambda-\vartheta(\x))\right\},\qquad
\x \in \O_{\mathrm{per}}.$$ Note that, for any $t \in \mathbb{R}$
and any $\x \in \O_{\mathrm{per}}$,
$\mathfrak{p}(\varphi_t(\x))=\mathfrak{p}(\x)$ so that
$$\vartheta(\varphi_t(\x))=\vartheta(\x), \quad \text{ and } \quad M_\lambda(\varphi_t(\x))=M_\lambda(\x) \qquad \forall \x \in \O_{\mathrm{per}}.$$
One has,
\begin{theo} For any $n \in \mathbb{N}$, $\sigma(\T_{\mathrm{per},n}) \subset \left\{\lambda \in
\mathbb{C}\,;\,1 \in \mathcal{R}_{\mathrm{ess}}(M_\lambda)\right\}$
where $\mathcal{R}_\mathrm{ess}(M_\lambda)$ denotes the essential
range of $M_\lambda\::\O_{\mathrm{per},n} \to \mathbb{C}.$
\end{theo}

\begin{proof} Let $\k \in \mathbb{N}$ be fixed.   First, for any
$\lambda \in \mathbb{C}$, let
$$\mathcal{J}_\lambda \::\:f \in X_{\mathrm{per},\k} \mapsto
\mathcal{J}_\lambda f(\x)=\int_0^{\mathfrak{p}(\x)} \exp(-\lambda t)
\u_{\mathrm{per},\k}(t)f(\x)\d t, \qquad \x \in
\O_{\mathrm{per},\k}.$$ The proof of the Proposition is based on the
fact that $\mathcal{J}_\lambda f \in \D(\T_{\mathrm{per},\k})$ for
any $f \in X_{\mathrm{per},\k}$ with
\begin{equation}\label{lambda-Tper}(\lambda-\T_{\mathrm{per},\k})\mathcal{J}_\lambda
f(\x)=\left(1-M_\lambda(\x)\right)f(\x), \qquad \x \in
\O_{\mathrm{per},\k}.\end{equation} Indeed, given $f \in X$, since
$\mathfrak{p}(\x) \leq n$ for any $\x \in \O_{\mathrm{per},\k}$, one
sees that $\mathcal{J}_\lambda f \in X_{\mathrm{per},\k}$ with
$$\|\mathcal{J}_\lambda f\|_{X_{\mathrm{per},\k}} \leq \int_0^n \exp(-\mathrm{Re}\lambda
t)\|\u_{\mathrm{per},\k}(t)f\|_{X_{\mathrm{per},\k}} dt.$$ Now,
given $t \in \mathbb{R}$ and $f \in X_{\mathrm{per},\k}$, since
$\mathfrak{p}(\cdot)$ is invariant under the action of the flow
$(\varphi_t)_t$, one sees easily that
\begin{equation*}\begin{split}
\u_{\mathrm{per},\k}(t)\mathcal{J}_\lambda
f(\x)&=\int_0^{\mathfrak{p}(\x)}\exp(-\lambda
s)\u_{\mathrm{per},\k}(t+s)f(\x)\d s\\
&=\exp(\lambda t) \int_t^{t+\mathfrak{p}(\x)}\exp(-\lambda
s)\u_{\mathrm{per},\k}(s)f(\x)\d s, \qquad \forall \x \in
\O_{\mathrm{per},\k},
\end{split}
\end{equation*}
and,
$$\dfrac{\d}{\d t} \left[\u_{\mathrm{per},\k}(t)\mathcal{J}_\lambda
f(\x)\right] \bigg|_{t=0}=\lambda \mathcal{J}_\lambda f(\x) +
 \exp\left(-\lambda
\mathfrak{p}(\x)\right)\u_{\mathrm{per},\k}(\mathfrak{p}(\x))f(\x)-f(\x),
\qquad \x \in \O_{\mathrm{per},\k}.$$ Now, since
$$[\u_{\mathrm{per},\k}(\mathfrak{p}(\x))f](\x)=\exp\left(-\int_0^{\mathfrak{p}(\x)}\nu(\varphi_s(\x))\d
s\right)f(\x)$$ one has  $$ \dfrac{\d}{\d t}
\left[\u_{\mathrm{per},\k}(t)\mathcal{J}_\lambda f(\x)\right]
\bigg|_{t=0}=\lambda \mathcal{J}_\lambda f(\x) +
\left(M_\lambda(\x)-1\right) f(\x), \qquad \forall \x \in
\O_{\mathrm{per},\k} .$$  This proves that $\mathcal{J}_\lambda f
\in \D(\T_{\mathrm{per},\k})$   and  Eq. \eqref{lambda-Tper} holds.
Now, let us prove that $\lambda \in \sigma(\T_{\mathrm{per},\k})$
implies $1 \in \mathcal{R}_\mathrm{ess}(M_\lambda)$. To do so,
assume   $1 \notin \mathcal{R}_\mathrm{ess}(M_\lambda)$,  and define
the operator
\begin{equation*}
\mathscr{R}_\lambda f(\x)
=(1-M_{\lambda}(\x))^{-1}\mathcal{J}_\lambda f(\x). \end{equation*}
Since $1 \notin \mathcal{R}_\mathrm{ess}(M_\lambda)$,  one sees that
$\mathscr{R}_\lambda$ is a bounded operator in
$X_{\mathrm{per},\k}$. Moreover, using the fact that
$\mathfrak{p}(\cdot)$ and $M_\lambda(\cdot)$ are invariant under the
flow $(\varphi_t)_t$, one easily sees that
$$\mathscr{R}_\lambda
\u_{\mathrm{per},\k}(t)=\u_{\mathrm{per},\k}(t)\mathscr{R}_\lambda,\qquad
\qquad \qquad t \in \mathbb{R}.$$
 Classically, this implies that
$\mathscr{R}_\lambda\D(\T_{\mathrm{per},\k}) \subset
\D(\T_{\mathrm{per},\k})$ and
$$\mathscr{R}_\lambda\T_{\mathrm{per},\k}f=\T_{\mathrm{per},\k}\mathscr{R}_\lambda
f,\qquad \qquad \forall f \in \D(\T_{\mathrm{per},\k}).$$ Using
again the fact that $M_\lambda(\cdot)$ is invariant under the flow
$(\varphi_t)_t$, one sees that
$\T_{\mathrm{per},\k}\mathscr{R}_\lambda
f=(1-M_\lambda(\cdot))^{-1}\T_{\mathrm{per},\k}\mathcal{J}_\lambda
f$ and Eq. \eqref{lambda-Tper} asserts that
\begin{equation*}\label{resolv}
\mathscr{R}_\lambda
(\lambda-\T_{\mathrm{per},\k})f=(\lambda-\T_{\mathrm{per},\k})\mathscr{R}_\lambda
f,\qquad \qquad \forall f \in
\D(\T_{\mathrm{per},\k}).\end{equation*} This proves that $\lambda
\in \varrho(\T_{\mathrm{per},\k})$ with
$(\lambda-\T_{\mathrm{per},\k})^{-1}=\mathscr{R}_\lambda.$\end{proof}

\begin{nb} We conjecture that the inclusion in the
above Theorem is an equality. More precisely, if there is some $T
\geq 0$ such that $\mathfrak{p}(\x) \leq T$ for any $\x \in \O$,
then, we conjecture that
\begin{equation}\label{ident}\lambda \in \sigma(\T_\mathrm{per}) \Longleftrightarrow 1 \in
\mathcal{R}_\mathrm{ess}(M_\lambda).\end{equation} The following
reasoning comforts us in our belief. If $1 \in
\mathcal{R}_\mathrm{ess}(M_\lambda)$, then, for any fixed $\epsilon
>0$, the set $\O_\epsilon=\{\x \in \O_{\mathrm{per},n}\,;\,|1-M_\lambda (\x)|<\epsilon\}$ is of non zero
measure. Let $f_\epsilon \in X_\mathrm{per}$ be such that
$\|f_\epsilon\|=1$ and Supp$f_\epsilon \subset \O_\epsilon.$ Let
$g_\epsilon=\mathcal{J}_\lambda f_\epsilon$. Then, Eq.
\eqref{lambda-Tper} implies that
$$(\lambda-\T_\mathrm{per})g_\epsilon(\x)=(1-M_\lambda(\x))f_\epsilon
(\x),$$ so that
$$\|(\lambda-\T_\mathrm{per})g_\epsilon\| \leq \epsilon.$$
To prove that $\lambda \in \sigma(\T_\mathrm{per})$, it would
suffice  to prove that $\|g_\epsilon\| \geq C >0$ for some constant
$C >0$ that does not depend on $\epsilon.$ We did not succeed in
proving this point. Notice however that the identity \eqref{ident}
holds true in space of continuous functions \cite{arendtgreiner}.
\end{nb}

The following Proposition provides a  complete picture of  the set
of $\lambda$ for which $1 \in \mathcal{R}_\mathrm{ess}(M_\lambda)$.
Its proof is inspired by similar calculations already used in the
study of $2D$ neutron transport equations \cite{ttsp}:
\begin{propo}\label{specfk} Assume that $\mathfrak{m}(\O_{\mathrm{per}})\neq
0$ and there exists some $T>0$ such that $\mathfrak{p}(\x) \leq T$
for any $\x \in \O_{\mathrm{per}}$. Then,
$$1 \in \mathcal{R}_\mathrm{ess}(M_\lambda) \quad
\text{ if and only if } \quad \lambda \in \overline{{{\bigcup_{k \in
\mathbb{Z}}
  \mathcal{R}_{\mathrm{ess}}(F_k)}}}$$ where, for any $k \in
\mathbb{Z}$, $F_k\;:\;\O_{\mathrm{per}} \to \mathbb{C}$ is a
measurable mapping given by
$$F_k(\x)=\vartheta(\x)+i\frac{2k \pi}{\mathfrak{p}(\x)}, \qquad \x \in
\O_{\mathrm{per}}.$$
\end{propo}
\begin{proof} Let us pick $\lambda \notin {\bigcup_{k \in \mathbb{Z}}
  \mathcal{R}_{\mathrm{ess}}(F_k)}$. Then, for any $k \in \mathbb{Z},$ there
exists $\beta_k
>0$ such that
$$|\lambda -F_k(\x)| \geq \beta_k \qquad \text{ a. e.  } \x \in \O_{\mathrm{per}},$$
i.e.
$$|-\mathfrak{p}(\x)\vartheta(\x)-2ik\pi -\lambda\,\mathfrak{p}(\x)|\geq
\mathfrak{p}(\x) \,\beta_k \geq 2\pi\beta_k/\kappa
 \qquad \text{ a. e.  } \x \in \O_{\mathrm{per}}$$
 where we made use of Yorke's Theorem for the last estimate. This means
that, for any integer $n \geq 0,$ there exists $c_n > 0$ such that
\begin{equation}\label{cn1}
\left|\dfrac{-\mathfrak{p}(\x)\vartheta(\x) -
\lambda\,\mathfrak{p}(\x)}{2\pi\,n} \pm i\right| \geq c_n   \quad
\text{ a. e. } \quad \x \in \O_{\mathrm{per}},\:n \geq
1,\end{equation} and
\begin{equation}\label{cn0}
|\mathfrak{p}(\x)\vartheta(\x)-\lambda\,\mathfrak{p}(\x)|\geq c_0
\qquad \text{ a. e.  } \x \in \O_{\mathrm{per}}.\end{equation} Now,
since
$$e^u-1=ue^{\frac{u}{2}}\prod_{n=1}^{\infty}\left(1+\dfrac{u^2}{4\pi^2\,n^2}\right),\qquad
u \in \mathbb{C},$$ and
$$\prod_{n=N}^{\infty}\left(1+\dfrac{u^2}{4\pi^2\,n^2}\right) \to 1 \quad (N
\to\infty)$$ uniformly on any compact subset of $\mathbb{C},$ for
any $M>0,$ there exists $N >0$ such that
$$\left|e^u-1\right| \geq
\dfrac{1}{2}\left|u\right|\left|e^{\frac{u}{2}}\right|\prod_{n=1}^{N}\left|i-\dfrac{u}{2\pi\,n}\right|\left|i+\dfrac{u}{2\pi\,n}\right|,\qquad
|u| < M.$$ Now, since $\nu$ is bounded and $\mathfrak{p}(\x) \leq
T$, there exists  $M
> 0$ (large enough), such that
$$\left|\mathfrak{p}(\x)\vartheta(\x)-\lambda\,\mathfrak{p}(\x)\right|\leq M \qquad \text{
a. e.  } \x \in \O_\mathrm{per},$$ one gets form \eqref{cn1} and
\eqref{cn0} that
$$|\exp\{\mathfrak{p}(\x)\vartheta(\x)-\lambda\,\mathfrak{p}(\x) \}-1| \geq
\dfrac{C}{2}\prod_{n=1}^Nc_n^2  \qquad \text{ a. e.  } \x \in
\O_\mathrm{per},$$ where $C=\esin_{\x \in
\O_{\mathrm{per}}}\left|\exp\left\{\dfrac{\mathfrak{p}(\x)}{2}(
\vartheta(\x)-\lambda )\right\}\right|.$ Hence, $\esin_{\x \in
\O_\mathrm{per}}|M_\lambda(\x)-1| >0$ which proves the first
inclusion. To prove the second inclusion, it is enough to notice,
from the continuity of the exponential function that, if there is
some constant $C >0$ such that
$$\left|\exp( \mathfrak{p}(\x)\vartheta(\x)-\lambda
\mathfrak{p}(\x))-1\right| \geq C >0, \qquad \text{ a. e. } \x \in
\O_\mathrm{per},$$ then, for any $k \in \mathbb{Z}$, there exists
$c_k >0$ such that $\left| \mathfrak{p}(\x)\vartheta(\x)-\lambda
\mathfrak{p}(\x)-2ik\pi\right| \geq c_k$ for a. e. $\x \in
\O_\mathrm{per}$. Now, since $\mathfrak{p}(\x) \leq T$ for any $\x
\in \O_\mathrm{per}$, one sees that $|\lambda-F_k(\x)| \geq c_k/T$
for a. e. $\x \in \O_\mathrm{per}$.\end{proof}

It remain to investigate the spectral properties of the group
$(\u_\mathrm{per}(t))_{t \in \mathbb{R}}$. As we already saw it, the
Spectral Mapping Theorem fails to be true in general. However, one
can deduce from Proposition A. \ref{modulelspec} the following
version of the Annular Hull Theorem:
 \begin{theo} Assume
 $\mathfrak{m}(\O_\mathrm{per})\neq 0.$ Then,
 \begin{equation}\label{annulaper} \mathbb{T} \cdot \sigma\big (\u_\mathrm{per}(t)\big ) =\mathbb{T} \cdot \exp\bigg(t
 \sigma \big(\T_\mathrm{per}\big) \cap \R\bigg)
, \qquad \forall t \in \mathbb{R }.
 \end{equation}
 In particular,  the semigroup $(\u(t))_{t \geq 0}$ fulfils the Annular Hull
 Theorem:
 \begin{equation*}
 \mathbb{T} \cdot \bigg(\sigma(\u(t)) \setminus \{0\}\bigg) =
 \mathbb{T} \cdot \exp\bigg(t\sigma(\T) \bigg), \qquad \forall t
 \geq 0.
 \end{equation*}
 \end{theo}
 \begin{nb} The above Annular Hull Theorem is known to be true for
 general weighted shift \textit{semigroups} \cite{lst,chicone}.
 Notice that the proof of \cite{lst,chicone} rely on completely
 different arguments and involve very sophisticated tools. On the
 contrary, our result is a very easy consequence of a more general
 result on positive groups on $L^p$ spaces. Of course, the
 simplification relies on the fact that we are dealing here with
 $C_0$-groups rather than semigroups.
 \end{nb}
 In the $L^1$ case, one can strengthen this result thanks to the  Weak Spectral
Mapping Theorem for positive $C_0$--group  by W. Arendt and G.
Greiner \cite[Corollary 1.4]{arendtgreiner}. Precisely,
\begin{theo}\label{L1}
Assume $p=1$, i.e.
$X_\mathrm{per}=L^1(\O_\mathrm{per},\d\mathfrak{m})$. For any $t \in
\mathbb{R}$, the following weak spectral mapping theorem holds:
 $\sigma(\u_{\mathrm{per}}(t))=\overline{\exp\left(t
\sigma(\T_{\mathrm{per}})\right)} $ and, under the assumption
\ref{hypcontinu},
$$\sigma(\u(t))=\overline{\exp\left(t \sigma(\T)\right)}, \qquad
\forall t \geq 0.$$
\end{theo}

\section*{Appendix: Spectral properties of general positive semigroups}
\setcounter{theo}{0}

 Let $\X$ be  a
complex Banach lattice with positive cone $\X_+$ and let $(T(t))_{t
\in \R}$ be a positive $C_0$-group in $\mathscr{B}(\X)$ with
generator $A$. Recall that the positivity of the group $(T(t))_{t
\in \R}$ means that $\X_+$ is invariant under $T(t)$ for any $t \in
\R$. We establish in this Appendix several abstract results on
(positive) $C_0$-groups we used in the paper. The key point is the
following spectral decomposition result for strongly continuous
groups of positive operators due to Arendt \cite[Theorem
4.2]{arendt} and Greiner \cite{greiner}.
\begin{theoA}\label{decompp}
Let $(T(t))_{t \in \R}$ be a $C_0$--group of positive operators with
generator $A$ on some Banach lattice $\X$. Let $\mu \in \varrho(A)
\cap \R.$ Then, $\X$ is the direct sum of the orthogonal projection
bands:
$$I_{\mu}=\{x \in \X\,;\,\mathscr{R}(\mu,A)|x| \geq 0\} \qquad \text{ and }
\qquad J_{\mu}=\{x \in \X\,;\,\mathscr{R}(\mu,A)|x| \leq 0\}.$$
Moreover, $I_{\mu}$ and $J_{\mu}$ are invariant under $T(t)$ $(t \in
\R)$ and $\sigma(A_{|I_{\mu}})=\{\lambda \in
\sigma(A)\,;\,\mathrm{Re} \lambda < \mu\}$,
$\sigma(A_{|J_{\mu}})=\{\lambda \in \sigma(A)\,;\,\mathrm{Re}
\lambda > \mu\}$ where $A_{|I_{\mu}}$ (respectively $A_{|J_{\mu}}$)
denotes the generator of $(T(t)_{|I_{\mu}})_{t \in \R}$ (resp. of
$(T(t)_{|J_{\mu}})_{t \in \R}$).
\end{theoA}

Using this result, G. Greiner \cite{greiner} has been able to prove
a \textit{Spectral Mapping Theorem} for the \textit{real spectrum}
(Theorem A. \ref{realspec} hereafter). Borrowing the ideas of
Greiner, it is possible to prove the following more general result.
Actually, we did not find this result in the literature and give
here a simple proof of it.
\begin{propoA}\label{modulelspec}
Let $(\Sigma,\varpi)$ be a $\sigma$-measured space and let
$(T(t))_{t \in \R}$ be a \emph{positive} $C_0$-group in
$L^p(\Sigma,\d \varpi)$ $(1 \leq p < \infty)$, with generator $A$.
Then,
\begin{equation*} \{|\mu|\,;\,\mu \in \sigma(T(t))\}= \exp{\{t
\left(\sigma(A) \cap \R\right)\}} \qquad \qquad \text{ for any } t
\geq 0.\end{equation*} In particular, the following Annular Hull
Theorem holds true:
\begin{equation}\label{annular}
\mathbb{T} \cdot \sigma(T(t)) = \mathbb{T} \cdot \exp\bigg(t
\big(\sigma(A) \cap \R\big)\bigg)= \mathbb{T} \cdot \exp\bigg(t
\sigma(A)\bigg), \qquad \forall t \in \mathbb{R}.
\end{equation}
\end{propoA}
\begin{proof} It is clear that $\exp{\{t
\left(\sigma(A) \cap \R\right)\}} \subset \{|\mu|\,;\,\mu \in
\sigma(T(t))\}.$ Let us show the converse inclusion. Let $z \in \R^+
\setminus \exp{\{t \left(\sigma(A) \cap \R\right)\}}.$ Then,
$z=\exp{(\alpha t)}$, with $\alpha \in \varrho(A) \cap \R.$ Then,
according to Theorem A. \ref{decompp}, there exist two projection
bands $I_{\alpha}$ and $J_{\alpha}$ such that $L^p(\Omega, \d
\varpi)=I_{\alpha} \oplus J_{\alpha}$ and
$$\sigma(A_{|I_{\alpha}})=\{z \in \sigma(A)\,;\,\mathrm{Re}
z < \alpha\},\qquad \text{ while } \qquad
\sigma(A_{|J_{\alpha}})=\{z \in \sigma(A)\,;\,\mathrm{Re} z
> \alpha\}.$$
Consequently, $s(A_{|I_{\alpha}}) < \alpha$ and $s(-A_{|J_{\alpha}})
< -\alpha$ where $s(\cdot)$ denotes the spectral bound. Since, in
$L^p$--spaces $(1 \leq p < \infty)$, the type of any positive
$C_0$-semigroup coincide with the spectral bound of its generator
\cite{weis}, one gets
$$r(T(t)_{|I_{\alpha}}) < \exp{(\alpha t)}, \qquad r(T(-t)_{|J_{\alpha}}) < \exp{(-\alpha t)}.$$ Hence, using
that $\sigma(T(t)_{|J_{\alpha}}) = \{z \in \mathbb{C}\,;\,1/z \in
\sigma(T(-t)_{|J_{\alpha}}) \}$,
$$\sigma(T(t)_{|I_{\alpha}}) \subset \{z \in \mathbb{C}\,;\,|z| <
\exp{(\alpha t)}\}\quad \text{ whereas } \quad
\sigma(T(t)_{|J_{\alpha}}) \subset \{z \in \mathbb{C}\,;\,|z|
> \exp{(\alpha t)}\}.$$  Now, since
$\sigma(T(t))=\sigma(T(t)_{|I_{\alpha}}) \cup
\sigma(T(t)_{|J_{\alpha}})$, any $\mu \in \mathbb{C}$ with
$|\mu|=\exp{(\alpha t)}$ is such that $\mu \notin \sigma(T(t))$
which achieves the proof. The proof of the Annular Hull Theorem is
then obvious since any $\mu \in \sigma(T(t))$ writes
$\mu=|\mu|\exp(i\theta)$ for some $\theta \in \R$ while
$|\mu|=\exp(\alpha t)$ for some $\alpha \in \sigma(A) \cap \R.$
\end{proof}

The \textit{Spectral Mapping Theorem} for the \textit{real spectrum}
of Greiner \cite{greiner} (see also \cite{arendt} and
\cite[Corollary 4.10]{nagel86}) is now a direct consequence of the
above Proposition.
\begin{theoA}\label{realspec}
Let $(\Sigma,\varpi)$ be a $\sigma$-measured space and let
$(T(t))_{t \in \R}$ be a \emph{positive} $C_0$-group in
$L^p(\Sigma,\d \varpi)$ $(1 \leq p < \infty)$, with generator $A$.
Then,
\begin{equation*} \sigma(T(t)) \cap \R_+ = \exp{\{t \left(\sigma(A)
\cap \R\right)\}} \qquad \qquad \text{ for any } t \geq
0.\end{equation*}
\end{theoA}
\begin{proof} It is clear that $\sigma(T(t)) \cap \R_+ \supseteq \exp{\{t \left(\sigma(A) \cap
\R\right)\}}$ for any $t \in \R$. Now, since $\sigma(T(t)) \cap \R_+
\subset \{|\mu|\,;\,\mu \in \sigma(T(t))\}$, the converse inclusion
follows immediately from Proposition A.\ref{modulelspec}.
\end{proof}

Another consequence of Proposition A. \ref{modulelspec} is the
following spectral mapping theorem which applies to generator whose
approximate spectrum is invariant by vertical translations:

\begin{theoA}\label{newspectral}
Let $(\Sigma,\varpi)$ be a $\sigma$-measured space and let
$(T(t))_{t \in \R}$ be a \emph{positive} $C_0$-group in
$L^p(\Sigma,\d \varpi)$ $(1 \leq p < \infty)$, with generator $A$.
If $\sigma_{\mathrm{ap}}(A)=\sigma_{\mathrm{ap}}(A)+i\R$ then
\begin{equation*} \sigma(T(t))= \exp{ \left( t \sigma(A)\right)}
\qquad \qquad \text{ for any } t \in \R.\end{equation*}
\end{theoA}
\begin{proof} It clearly suffices to prove the "$\subset$"
inclusion. We first note that $\sigma(A)=\sigma(A)+i\mathbb{R}.$
Indeed, let $\lambda \in \sigma(A)$, assume that $\lambda+i\R
\nsubseteq \sigma(A).$ Then, there is $\alpha \in \R$ such that
$\lambda + i\alpha$ lies in the boundary of $\sigma(A)$. In
particular, $\lambda +i\alpha \in \sigma_{\mathrm{ap}}(A)$, and by
assumption, $\lambda+i\R \subset \sigma(A)$ which is a
contradiction. Therefore, $\sigma(A)=\sigma(A)+i\R.$ Now, let $z
\notin \exp{ \left( t \sigma(A)\right)}.$ Then, there is $\lambda
\in \varrho(A)$ such that $z=\exp{(\lambda t)}$,
$\lambda=\alpha+i\beta.$ Since the spectrum of $A$ is invariant by
vertical translations, $\alpha \in \varrho(A)$. Hence $\exp{(\alpha
t)} \notin \exp{\{t \left(\sigma(A) \cap \R\right)\}}$. According to
Proposition A.\ref{modulelspec}, and since $|z|=\exp{(\alpha t)}$,
this means that $z \in \varrho(T(t))$.
\end{proof}

\begin{nbA} As a consequence of the above result, one sees that, if
$(T(t))_{t \in \R}$ is a positive groups of operator with generator
$A$ in some $L^p$-space ($1 \leq p < \infty$) and if
$\sigma(A)=\sigma(A)+i\R$, then $\sigma(T(t))=\sigma(T(t)) \cdot
\mathbb{T}$ $(t \in \R)$ where $\mathbb{T}$ is the unit circle of
$\mathbb{C}$.\end{nbA}

\end{document}